\documentclass[12pt,leqno,a4paper]{amsart}
\usepackage{amssymb,enumerate}
\newcommand{\FF}{{\mathbb{F}}}
\newcommand{\ZZ}{{\mathbb{Z}}}

\newcommand{\fS}{{\mathfrak{S}}}

\newcommand{\ba}{{\mathbf{a}}}
\newcommand{\bb}{{\mathbf{b}}}
\newcommand{\bC}{{\mathbf{C}}}
\newcommand{\bG}{{\mathbf{G}}}
\newcommand{\tbG}{{\tilde\bG}}
\newcommand{\bH}{{\mathbf{H}}}
\newcommand{\bK}{{\mathbf{K}}}
\newcommand{\bL}{{\mathbf{L}}}
\newcommand{\tbL}{{\tilde\bL}}
\newcommand{\bM}{{\mathbf{M}}}
\newcommand{\bN}{{\mathbf{N}}}
\newcommand{\bP}{{\mathbf{P}}}
\newcommand{\bS}{{\mathbf{S}}}
\newcommand{\bT}{{\mathbf{T}}}
\newcommand{\bX}{{\mathbf{X}}}
\newcommand{\bY}{{\mathbf{Y}}}
\newcommand{\bU}{{\mathbf{U}}}

\newcommand{\cE}{{\mathcal{E}}}
\newcommand{\cL}{{\mathcal{L}}}

\newcommand{\ad}{{\operatorname{ad}}}
\newcommand{\diag}{{\operatorname{diag}}}
\newcommand{\Ind}{{\operatorname{Ind}}}
\newcommand{\Irr}{{\operatorname{Irr}}}
\newcommand{\Res}{{\operatorname{Res}}}

\newcommand{\GL}{{\operatorname{GL}}}
\newcommand{\PGL}{{\operatorname{PGL}}}
\newcommand{\SL}{{\operatorname{SL}}}
\newcommand{\GU}{{\operatorname{GU}}}
\newcommand{\SU}{{\operatorname{SU}}}

\newcommand\RLG{{R_\bL^\bG}}
\newcommand\RMG{{R_\bM^\bG}}
\newcommand\RTG{{R_\bT^\bG}}
\newcommand\sRLG{{{}^*\!R_\bL^\bG}}
\newcommand\sRMG{{{}^*\!R_\bM^\bG}}
\newcommand\sRTG{{{}^*\!R_\bT^\bG}}
\newcommand\sRTL{{{}^*\!R_\bT^\bL}}
\newcommand\sRXL{{{}^*\!R_\bX^\bL}}
\newcommand\Zc{{Z^\circ}}

\newcommand{\tw}[1]{{}^#1\!}
\newcommand{\Ph}[1]{\Phi_#1}

\let\eps=\epsilon
\let\la=\lambda

\newtheorem{thm}{Theorem}[section]
\newtheorem{lem}[thm]{Lemma}

\newtheorem{prop}[thm]{Proposition}

\newtheorem*{thmA}{Theorem A}
\newtheorem*{thmB}{Theorem B}

\theoremstyle{definition}
\newtheorem{defn}[thm]{Definition}
\newtheorem{exmp}[thm]{Example}

\theoremstyle{remark}
\newtheorem{rem}[thm]{Remark}

\begin{document}

\title[Lusztig induction and $\ell$-blocks]{Lusztig induction and $\ell$-blocks\\ of finite reductive groups}

\date{\today}

\author{Radha Kessar}
\address{Department of Mathematical Sciences, City University London,
         Northampton Square, London EC1V 01B, U.K.}
\email{Radha.Kessar.1@city.ac.uk}
\author{Gunter Malle}
\address{FB Mathematik, TU Kaiserslautern, Postfach 3049,
         67653 Kaisers\-lautern, Germany.}
\email{malle@mathematik.uni-kl.de}

\thanks{The second author gratefully acknowledges financial support by ERC
  Advanced Grant 291512.}



\dedicatory{To the memory of Robert Steinberg}

\begin{abstract}
We present a unified parametrisation of $\ell$-blocks of quasi-simple finite
groups of Lie type in non-defining characteristic via Lusztig's induction
functor in terms of $e$-Jordan-cuspidal pairs and $e$-Jordan quasi-central
cuspidal pairs.
\end{abstract}

\maketitle


\section{Introduction }   \label{sec:intro}
The work of Fong and Srinivasan for classical matrix groups and of Schewe for
certain blocks of groups of exceptional type exhibited a close relation
between the $\ell$-modular block structure of groups of Lie type and the
decomposition of Lusztig's induction functor, defined in terms of $\ell$-adic
cohomology.
This connection was extended to
unipotent blocks of arbitrary finite reductive groups and large primes $\ell$
by Brou\'e--Malle--Michel \cite{BMM}, to all unipotent blocks by
Cabanes--Enguehard \cite{CE94} and Enguehard \cite{En00}, to groups with
connected center and primes $\ell\ge7$ by Cabanes--Enguehard \cite{CE99}, to
non-quasi-isolated blocks by Bonnaf\'e--Rouquier \cite{BR03} and to
quasi-isolated blocks of exceptional groups at bad primes by the authors
\cite{KM}.  \par
It is the main purpose of this paper to unify and extend all of the preceding
results in particular from \cite{CE99} so as
to establish a statement in its largest possible generality, without
restrictions on the prime $\ell$, the type of group or the type of block, in
terms of $e$-Jordan quasi-central cuspidal pairs (see Section~\ref{sec:cusp}
for the notation used).

\begin{thmA}   \label{thmA}
 Let $\bH$ be a simple algebraic group of simply connected type with a
 Frobenius endomorphism $F:\bH\rightarrow\bH$ endowing $\bH$ with an
 $\FF_q$-rational structure. Let $\bG$ be an $F$-stable Levi subgroup of $\bH$.
 Let $\ell$ be a prime not dividing $q$ and set $e=e_\ell(q)$.
 \begin{enumerate}[\rm(a)]
  \item For any $e$-Jordan-cuspidal pair $(\bL,\la)$ of $\bG$ such that
   $\la\in\cE(\bL^F,\ell')$, there exists a unique $\ell$-block
   $b_{\bG^F}(\bL,\la)$ of $\bG^F$ such that all irreducible constituents
   of $\RLG(\la)$ lie in $b_{\bG^F}(\bL,\la)$.
  \item The map $\Xi:(\bL,\la)\mapsto b_{\bG^F}(\bL,\la)$ is a
   surjection from the set of $\bG^F$-conjugacy classes of $e$-Jordan-cuspidal
   pairs $(\bL,\la)$ of $\bG$ such that $\la \in \cE(\bL^F,\ell')$ to the set
   of $\ell$-blocks of $\bG^F$.
  \item The map $\Xi$ restricts to a surjection from the set of
   $\bG^F$-conjugacy classes of $e$-Jordan quasi-central cuspidal pairs
   $(\bL,\la)$ of $\bG$ such that $\la \in \cE(\bL^F,\ell')$ to the set
   of $\ell$-blocks of $\bG^F$.
  \item For $\ell\ge3$ the map $\Xi$ restricts to a bijection between the set
   of $\bG^F$-conjugacy classes of $e$-Jordan quasi-central cuspidal pairs
   $(\bL,\la)$ of $\bG$ with $\la \in \cE(\bL^F,\ell')$ and the set of
   $\ell$-blocks of $\bG^F$.
  \item The map $\Xi$ itself is bijective if $\ell\geq 3$ is good for $\bG$,
   and $\ell\ne 3$ if $\bG^F$ has a factor $\tw3D_4(q)$.
 \end{enumerate}
\end{thmA}

The restrictions in (d) and~(e) are necessary (see Remark~\ref{rem:part(e)} and
Example~\ref{exmp:part(d)}).

In fact, part~(a) of the preceding result is a special case of the
following characterisation of the $\ell'$-characters in a given $\ell$-block
in terms of Lusztig induction:

\begin{thmB}   \label{thmB}
 In the setting of Theorem~A let $b$ be an $\ell$-block of $\bG^F$ and denote
 by $\cL(b)$ the set of $e$-Jordan cuspidal pairs $(\bL,\la)$ of $\bG$ such that
 $\Irr(b)\cap\RLG(\la)\ne\emptyset$. Then
 $$\Irr(b)\cap\cE(\bG^F,\ell')
   =\{\chi\in\cE(\bG^F,\ell')\mid \exists\,(\bL,\la)\in\cL(b)\text{ with }
    (\bL,\la)\ll_e(\bG,\chi)\}.$$
\end{thmB}

Note that at present, it is not known whether Lusztig induction $\RLG$ is
independent
of the parabolic subgroup containing the Levi subgroup $\bL$ used to define it.
Our proofs will show, though, that in our case $b_{\bG^F}(\bL,\la)$ is defined
unambiguously. \par

An important motivation for this work comes from the recent reductions of
most long-standing famous conjectures in modular representation theory of
finite groups to questions about quasi-simple groups. Among the latter, the
quasi-simple groups of Lie type form the by far most important part. A
knowledge and suitable inductive description of the $\ell$-blocks of these
groups is thus of paramount importance for an eventual proof of those
central conjectures. Our results are specifically tailored for use in an
inductive approach by considering groups that occur as Levi subgroups inside
groups of Lie type of simply connected type, that is, inside quasi-simple
groups.
\medskip

Our paper is organised as follows; in Section~\ref{sec:cusp}, we set up
$e$-Jordan (quasi-central) cuspidal pairs and
discuss some of their properties. In Section~\ref{sec:main} we prove
Theorem~A (see Theorem~\ref{thm:ecusp-blocks}) on parametrising $\ell$-blocks
by $e$-Jordan-cuspidal and $e$-Jordan quasi-central cuspidal pairs and
Theorem~B (see Theorem~\ref{thm:all covered}) on characterising
$\ell'$-characters in blocks. The crucial case turns out to be when $\ell=3$.
In particular, the whole Section~\ref{subsec:An} is devoted to the situation
of extra-special defect groups of order~27, excluded in \cite{CE99}, which
eventually turns out to behave just as the generic case. An important
ingredient of Section~\ref{sec:main} is
Theorem~\ref{thm:all e-splits}, which shows that the distribution of
$\ell'$-characters in $\ell$-blocks is preserved under Lusztig induction
from $e$-split Levi subgroups. Finally, in Section~\ref{sec:Jordan} we collect
some results relating $e$-Jordan-cuspidality and usual $e$-cuspidality.

\section{Cuspidal pairs}   \label{sec:cusp}

Throughout this section, $\bG$ is a connected reductive linear algebraic group
over the algebraic closure of a finite field of characteristic~$p$, and
$F:\bG\rightarrow\bG$ is a Frobenius endomorphism endowing $\bG$ with an
$\FF_q$-structure for some power $q$ of~$p$. By $\bG^*$ we denote a group
in duality with $\bG$ with respect to some fixed $F$-stable maximal torus
of $\bG$, with corresponding Frobenius endomorphism also denoted by $F$.

\subsection{$e$-Jordan-cuspidality}
Let $e$ be a positive integer. We will make use of the terminology of Sylow
$e$-theory (see e.g.~\cite{BMM}). For an $F$-stable maximal torus $\bT$,
$\bT_e$ denotes its Sylow $e$-torus. Then a Levi subgroup $\bL\le\bG$ is
called \emph{$e$-split} if $\bL=C_\bG(\Zc(\bL)_e)$, and $\la\in\Irr(\bL^F)$
is called \emph{$e$-cuspidal} if $^*\!R_{\bM\le\bP}^\bL(\la)=0$ for all proper
$e$-split Levi subgroups $\bM<\bL$ and any parabolic subgroup $\bP$ of $\bL$
containing $\bM$ as Levi complement.
(It is expected that Lusztig induction is in fact independent of the ambient
parabolic subgroup. This would follow for example if the Mackey formula holds
for $\RLG$, and has been proved whenever $\bG^F$ does not have any component
of type $\tw2E_6(2),E_7(2)$ or $E_8(2)$, see \cite{BM11}. All the statements
made in this section using $\RLG$ are valid independent of the particular
choice of parabolic subgroup --- we will make clarifying remarks at points
where there might be any ambiguity.)

\begin{defn}
Let $s\in{\bG^*}^F$ be semisimple. Following \cite[1.3]{CE99} we say that
$\chi\in\cE(\bG^F,s)$ is \emph{$e$-Jordan-cuspidal}, or \emph{satisfies
condition~(J)} with respect to some $e\ge1$ if
\begin{enumerate}
\item[(J$_1$)] $\Zc(C_{\bG^*}^\circ(s))_e=\Zc(\bG^*)_e$ and
\item[(J$_2$)] $\chi$ corresponds under Jordan decomposition to the
 $C_{\bG^*}(s)^F$-orbit of an $e$-cuspidal unipotent character of
 $C_{\bG^*}^\circ(s)^F$.
\end{enumerate}
If $\bL\le\bG$ is $e$-split and $\la\in\Irr(\bL^F)$ is $e$-Jordan-cuspidal,
then $(\bL,\la)$ is called an \emph{$e$-Jordan-cuspidal pair}.
\end{defn}

It is shown in \cite[Prop.~1.10]{CE99} that $\chi$ is $e$-Jordan-cuspidal if
and only if it satisfies the \emph{uniform criterion}
$$\text{(U): for every $F$-stable maximal torus $\bT\le\bG$ with
  $\bT_e\not\le Z(\bG)$ we have $\sRTG(\chi)=0$}.$$

\begin{rem}   \label{rem:KM}
By \cite[Prop.~1.10(ii)]{CE99} it is known that $e$-cuspidality implies
$e$-Jordan-cuspi\-dality; moreover $e$-Jordan-cuspidality and
$e$-cuspidality agree at least in the following situations:
\begin{enumerate}
 \item when $e=1$;
 \item for unipotent characters (see \cite[Cor.~3.13]{BMM});
 \item for characters lying in an $\ell'$-series where $\ell$ is good
  for $\bG$ and either $\ell\ge5$ or $\ell=3\in\Gamma(\bG,F)$ (see
  \cite[Thm.~4.2 and~Rem.~5.2]{CE99}); and
 \item for characters lying in a quasi-isolated $\ell'$-series of
  an exceptional type simple group for $\ell$ a bad prime (this follows by
  inspection of the explicit results in \cite{KM}).
\end{enumerate}
To see the the first point, assume that $\chi$ is 1-Jordan-cuspidal. Suppose
if possible that $\chi$ is not 1-cuspidal. Then there exists a proper 1-split
Levi subgroup $\bL$ of $\bG$ such that $\sRLG(\chi)$ is non-zero. Then
$\sRLG(\chi)(1)\ne 0$ as $\sRLG$ is ordinary Harish-Chandra restriction.
Hence the projection of $\sRLG(\chi)$ to the space of uniform functions of
$\bL^F$ is non-zero in contradiction to the uniform criterion (U).

It seems reasonable to expect (and that is formulated as a conjecture in
\cite[1.11]{CE99}) that $e$-cuspidality and $e$-Jordan-cuspidality agree
in general.
See Section~\ref{sec:Jordan} below for a further discussion of this.
\end{rem}

We first establish conservation of $e$-Jordan-cuspidality under some
natural constructions:

\begin{lem}   \label{lem:derived-Jcusp}
 Let $\bL$ be an $F$-stable Levi subgroup of $\bG$ and $\la\in\Irr(\bL^F)$.
 Let $\bL_0 =\bL\cap [\bG,\bG]$ and let $\la_0$ be an irreducible constituent
 of $\Res^{\bL^F}_{\bL_0^F}(\la)$. Let $e\ge1$.
 Then $(\bL,\la)$ is an $e$-Jordan-cuspidal pair for $\bG$ if and only if
 $(\bL_0,\la_0)$ is an $e$-Jordan-cuspidal pair for $[\bG,\bG]$.
\end{lem}

\begin{proof}
Note that $\bL$ is $e$-split in $\bG$ if and only if $\bL_0$ is $e$-split in
$\bG_0$. Let $\iota:\bG\hookrightarrow\tbG$ be a regular embedding. It is
shown in the proof of \cite[Prop.~1.10]{CE99} that condition (J) with respect
to $\bG$ is equivalent to condition (J) with respect to $\tbG$. Since $\iota$
restricts to a regular embedding $[\bG,\bG]\hookrightarrow\tbG$, the
same argument shows that condition (J) with respect to $\tbG$ is
equivalent to that condition with respect to $[\bG,\bG]$.
\end{proof}

\begin{prop}   \label{prop:biject-Jcusp}
 Let $s\in {\bG^*}^F$ be semisimple, and $\bG_1\le\bG$ an $F$-stable Levi
 subgroup with $\bG_1^*$ containing
 $C_{\bG^*}(s)$. For $(\bL_1,\la_1)$ an $e$-Jordan-cuspidal pair of $\bG_1$
 below $\cE(\bG_1^F,s)$ define $\bL:=C_\bG(\Zc(\bL_1)_e)$ and
 $\la:=\eps_\bL\eps_{\bL_1}R_{\bL_1}^\bL(\la_1)$. Then
 $\Zc(\bL_1)_e=\Zc(\bL)_e$, and  $(\bL_1,\la_1)\mapsto(\bL,\la)$ defines a
 bijection $\Psi_{\bG_1}^\bG$ between the set of $e$-Jordan-cuspidal pairs
 of $\bG_1$ below $\cE(\bG_1^F,s)$ and the set of $e$-Jordan-cuspidal pairs
 of $\bG$ below $\cE(\bG^F,s)$.
\end{prop}

We note that the character $\la $ and hence the bijection $\Psi_{\bG_1}^\bG$
above are independent of the choice of parabolic subgroup. This is explained
in the proof below.

\begin{proof}
We first show that $\Psi_{\bG_1}^\bG$ is well-defined. Let $(\bL_1,\la_1)$ be
$e$-Jordan-cuspidal in $\bG_1$ below $\cE(\bG_1^F,s)$, so $s\in\bL_1^*$. Then
$\bL^*:=C_{\bG^*}(\Zc(\bL_1^*)_e)$ clearly is an $e$-split Levi subgroup of
$\bG^*$. Moreover we have
$$\bL_1^*=C_{\bG_1^*}(\Zc(\bL_1^*)_e)=C_{\bG^*}(\Zc(\bL_1^*)_e)\cap \bG_1^*
  =\bL^*\cap\bG_1^*.$$
Now $s\in\bL_1^*$ by assumption, so
$$\bL_1^*=\bL^*\cap\bG_1^*\ge \bL^*\cap C_{\bG^*}(s)=C_{\bL^*}(s).$$
In particular, $\bL_1^*$ and $\bL^*$ have a maximal torus in common, so
$\bL_1^*$ is a Levi subgroup of $\bL^*$. Thus, passing to duals, $\bL_1$ is
a Levi subgroup of $\bL=C_\bG(\Zc(\bL_1)_e)$.   \par
We clearly have $\Zc(\bL_1)_e\le \Zc(\bL)_e$. For the reverse inclusion note
that $\Zc(\bL)_e\le\bL_1$, as $\bL_1$ is a Levi subgroup in $\bL$, so indeed
$\Zc(\bL)_e\le \Zc(\bL_1)_e$. \par
Hence $\la:=\eps_\bL\eps_{\bL_1}R_{\bL_1}^\bL(\la_1)$ is irreducible since,
as we saw above, $\bL_1^*\ge C_{\bL^*}(s)$. By \cite[Rem.~13.28]{DM91},
$\la$ is independent of the choice of parabolic
subgroup of $\bL$ containing $\bL_1$ as Levi subgroup. Let's argue
that $\la$ is $e$-Jordan-cuspidal. Indeed, for any $F$-stable maximal torus
$\bT\le\bL$ we have by the Mackey-formula (which holds as one of the Levi
subgroups is a maximal torus by a result of Deligne--Lusztig, see
\cite[Thm.(2)]{BM11}) that
$\eps_\bL\eps_{\bL_1}\sRTL(\la)={}\sRTL R_{\bL_1}^\bL(\la_1)$ is a sum of
$\bL^F$-conjugates of ${}^*\!R_{\bT}^{\bL_1}(\la_1)$. As $\la_1$ is
$e$-Jordan-cuspidal, this vanishes if $\bT_e\not\le \Zc(\bL_1)_e=\Zc(\bL)_e$.
So $\la$ satisfies condition~(U), hence is $e$-Jordan-cuspidal, and
$\Psi_{\bG_1}^\bG$ is well-defined.
\par
It is clearly injective, since if
$(\bL,\la)=\Psi_{\bG_1}^\bG(\bL_2,\la_2)$ for some $e$-cuspidal pair
$(\bL_2,\la_2)$ of $\bG_1$, then $\Zc(\bL_1)_e=\Zc(\bL)_e=\Zc(\bL_2)_e$, whence
$\bL_1=C_{\bG_1}(\Zc(\bL_1)_e)=C_{\bG_1}(\Zc(\bL_2)_e)=\bL_2$, and then the
bijectivity of $R_{\bL_1}^\bL$ on $\cE(\bL_1^F,s)$ shows that $\la_1=\la_2$ as
well. \par
We now construct an inverse map. For this let $(\bL,\la)$ be an
$e$-Jordan-cuspidal pair of $\bG$ below $\cE(\bG^F,s)$, and $\bL^*\le\bG^*$
dual to $\bL$. Set
$$\bL_1^*:=C_{\bG_1^*}(\Zc(\bL^*)_e)=C_{\bG^*}(\Zc(\bL^*)_e)\cap\bG_1^*
  =\bL^*\cap\bG_1^*,$$
an $e$-split Levi subgroup of $\bG_1^*$. Note that $s\in\bL^*$, so there
exists some maximal torus $\bT^*$ of $\bG^*$ with
$\bT^*\le C_{\bG^*}(s)\le\bG_1^*$, whence $\bL_1^*$ is a Levi subgroup of
$\bL^*$. Now again
$$\bL_1^*=\bL^*\cap\bG_1^*\ge\bL^*\cap C_{\bG^*}(s)=C_{\bL^*}(s).$$
So the dual $\bL_1:=C_{\bG_1}(\Zc(\bL)_e)$ is a Levi subgroup of $\bL$ such
that $\eps_{\bL_1}\eps_\bL R_{\bL_1}^\bL$ preserves irreducibility on
$\cE(\bL_1^F,s)$. We define $\la_1$ to be the unique constituent of
$^*R_{\bL_1}^\bL(\la)$ in the series $\cE(\bL_1^F,s)$. Then $\la_1$ is
$e$-Jordan-cuspidal. Indeed, for any $F$-stable maximal torus $\bT\le\bL_1$
with $\bT_e\not\le \Zc(\bL)_e=\Zc(\bL_1)_e$ we get that
${}^*\!R_\bT^{\bL_1}(\la_1)$ is a constituent of $\sRTL(\la)=0$ by
$e$-Jordan-cuspidality of $\la$. Here note that the set of constituents of
${}^*\!R_\bT^{\bL_1}(\eta)$, where $\eta$ is a constituent of
${}^*\!R_{\bL_1}^{\bL}(\la)$ different from $\la_1$ is disjoint from the set
of irreducible constituents of ${}^*\!R_\bT^{\bL_1}(\la_1)$.
\par
Thus we have obtained a well-defined map $^*\Psi_{\bG_1}^\bG$ from
$e$-Jordan-cuspidal pairs in $\bG$ to $e$-Jordan-cuspidal pairs in $\bG_1$,
both below the series $s$.
As the map $\Psi_{\bG_1}^\bG$ preserves the $e$-part of the center,
$^*\Psi_{\bG_1}^\bG\circ\Psi_{\bG_1}^\bG$ is the identity. It remains to prove
that $\Psi_{\bG_1}^\bG$ is surjective. For this let $(\bM,\mu)$ be any
$e$-Jordan-cuspidal pair of $\bG$ below $\cE(\bG^F,s)$, let
$(\bL_1,\la_1)={}^*\Psi_{\bG_1}^\bG(\bM,\mu)$ and
$(\bL,\la)=\Psi_{\bG_1}^\bG(\bL_1,\la_1)$. Then we have
$\Zc(\bM)_e\le \Zc(\bL_1)_e=\Zc(\bL)_e$, so
$\bL=C_\bG(\Zc(\bL)_e)\le C_\bG(\Zc(\bM)_e)=\bM$ is an $e$-split Levi subgroup
of $\bM$. As $\bL_1\le\bL\le\bM$ and $\eps_{\bL_1}\eps_\bM R_{\bL_1}^\bM$ is a
bijection from $\cE(\bL_1^F,s)$ to $\cE(\bM^F,s)$, it follows that
$\eps_\bL\eps_\bM R_\bL^\bM$ is a bijection between $\cE(\bL^F,s)$ and
$\cE(\bM^F,s)$. As $\la$ and $\mu$ are $e$-Jordan-cuspidal, (J$_1$) implies
that $\Zc(\bM^*)_e=\Zc(\bL^*)_e$, so $\bM=\bL$, that is,
$(\bM,\mu)$ is in the image of $\Psi_{\bG_1}^\bG$.
The proof is complete.
\end{proof}

The above bijection also preserves relative Weyl groups.

\begin{lem}   \label{lem:biject-weyl}
 In the situation and notation of Proposition~\ref{prop:biject-Jcusp} let
 $(\bL,\la)=\Psi_{\bG_1}^\bG(\bL_1,\la_1)$. Then
 $N_{\bG_1^F}(\bL_1,\la_1)\le N_{\bG^F}(\bL,\la)$ and this
 inclusion induces an isomorphism of relative Weyl groups
 $W_{\bG_1^F}(\bL_1,\la_1)\cong W_{\bG^F}(\bL,\la)$.
\end{lem}

\begin{proof}
Let $g\in N_{\bG_1^F}(\bL_1,\la_1)$. Then $g$ normalises $\Zc(\bL_1)_e$ and
hence also  $\bL=C_\bG(\Zc(\bL_1)_e)$. Thus,
$$\,^g\la = \eps_{\bL_1}\eps_\bL R_{\,^g\bL_1}^{\,^g\bL} (\,^g\la_1)
          = \eps_{\bL_1}\eps_\bL R_{\bL_1}^\bL(\la_1) =\la $$
and the first assertion follows.

For the second assertion let $g\in N_{\bG^F} (\bL, \la)$ and let
$\bT$ be an $F$-stable maximal torus of $\bL_1$ and $\theta$ an irreducible
character of $\bT^F$ such that $\la_1$ is a constituent of
$R_{\bT}^{\bL_1}(\theta)$. Since $\la_1\in \cE(\bL_1^F, s)$, $(\bT, \theta)$
corresponds via duality (between $\bL_1$ and $\bL_1^*$) to the
$\bL_1^{*F}$-class of $s$, and all constituents of $R_{\bT}^{\bL_1}(\theta)$
are in $\cE(\bL_1^F, s)$. Consequently, $R_{\bL_1}^\bL$ induces a bijection
between the set of constituents of  $R_{\bT}^{\bL_1} (\theta)$  and the
set of constituents of  $R_{\bT}^\bL (\theta)$. In particular, $\la$ is
a constituent of $R_{\bT}^\bL (\theta)$. Since $g$ stabilises
$\la$, $\la$ is also a constituent of $R_{\,^g\bT}^\bL (\,^g\theta) $.
Hence $(\bT, \theta)$  and $\,^g(\bT, \theta)$  are geometrically conjugate
in $\bL$. Let $l\in\bL$ geometrically conjugate $\,^g(\bT,\theta)$
to $(\bT, \theta)$.
Since $C_{\bG^*}(s) \leq  \bG_1^*$, $lg\in \bG_1$ (see for instance
\cite[Lemma~7.5]{KM}). Hence $F(l)l^{-1}= F(lg)(lg)^{-1}\in\bG_1\cap\bL=\bL_1$.
By the Lang--Steinberg theorem applied to $\bL_1$, there exists
$l_1\in \bL_1$ such that $l_1l \in \bL^F $. Also, since $l_1 \in\bG_1$ and
$g \in \bG^F$, $l_1lg \in \bG_1^F$.
Thus, up to replacing $g$ by $l_1lg$, we may assume that $g\in\bG_1^F$.

Since $\bL_1=C_{\bG_1}(\Zc(\bL)_e)$, it follows that $g\in N_{\bG_1^F}(\bL_1)$,
and thus
$$\eps_{\bL_1}\eps_\bL R_{\bL_1}^\bL(\la_1) =\la = \,^g\la
   =\eps_{\bL_1}\eps_\bL R_{\bL_1}^\bL (\,^g \la_1).$$
Since $R_{\bL_1}^\bL$ induces a bijection between the set of characters in
the geometric Lusztig series of $\bL_1^F$ corresponding to $s$ (the union of
series $\cE (\bL_1^F,t)$, where $t$ runs over the semisimple elements of
$\bL_1^{*F}$ which are $\bL_1$-conjugate to $s$) and the set of characters
in the geometric Lusztig series of $\bL^F$ corresponding to $s$, it suffices
to prove that $\,^g\la_1 \in \cE(\bL_1^F,t)$ for some $t\in \bL_1^{*F}$
which is ${\bL_1^*}^F$-conjugate to $s$.
Let $\bT $, $\theta$ and $l$ be as above. Since $lg\in\bG_1$ and
$g\in\bG_1$, it follows that $l\in\bG_1\cap \bL=\bL_1$. Hence
$\,^g(\bT,\theta)$ and $(\bT,\theta)$ are geometrically conjugate in $\bL_1$.
The claim follows as $\,^g\lambda_1$ is a constituent of
$R_{^g\bT}^{\bL_1}({^g\theta})$.
\end{proof}

\subsection{$e$-Jordan-cuspidality and $\ell$-blocks}

We next investigate the behaviour of $\ell$-blocks with respect to the
map $\Psi_{\bG_1}^\bG$. For this, let $\ell\ne p$ be a prime. We set
$$e_\ell(q):=\text{order of $q$ modulo }\begin{cases} \ell&
    \text{if }\ell\ne2\\   4& \text{if }\ell=2.\end{cases}$$
For a semisimple $\ell'$-element $s$ of $\bG^{*F}$, we denote by
$\cE_\ell(\bG^F,s)$ the union of all Lusztig series $\cE(\bG^F,st)$, where 
$t\in \bG^{*F}$ is an $\ell$-element commuting with $s$. We recall that the
set $\cE_\ell(\bG^F, s)$ is a union of $\ell$-blocks. Further, if
$\bG_1\le\bG$ is an $F$-stable Levi subgroup  such that $\bG_1^*$ contains
$C_{\bG^*}(s)$, then $\epsilon_{\bG_1}\epsilon_\bG R_{\bG_1}^\bG$ induces a
bijection, which we refer to as the Bonnaf\'e--Rouquier correspondence,
between the $\ell$-blocks in $\cE(\bG_1^F,s)$ and the $\ell$-blocks in
$\cE(\bG^F, s)$.

\begin{prop}   \label{prop:biject-Jcusp2}
 Let $\ell\ne p$ be a prime, $s\in {\bG^*}^F$ an $\ell'$-element and
 $\bG_1\le\bG$ an $F$-stable Levi subgroup with
 $\bG_1^*$ containing $C_{\bG^*}(s)$. Assume that $b$ is an $\ell$-block in
 $\cE_\ell(\bG^F,s)$, and $c$ is its Bonnaf\'e--Rouquier correspondent in
 $\cE_\ell(\bG_1^F,s)$. Let $e:=e_\ell(q)$.
 \begin{itemize}
  \item[\rm(a)] Let $(\bL_1,\la_1)$ be $e$-Jordan-cuspidal in $\bG_1$ and set
   $(\bL,\la)=\Psi_{\bG_1}^\bG(\bL_1,\la_1)$. If all constituents of
   $R_{\bL_1}^{\bG_1}(\la_1)$ lie in $c$, then all constituents of $\RLG(\la)$
   lie in $b$.
  \item[\rm(b)] Let $(\bL,\la)$ be $e$-Jordan-cuspidal in $\bG$ and set
   $(\bL_1,\la_1)={}^*\Psi_{\bG_1}^\bG(\bL,\la)$. If all constituents of
   $\RLG(\la)$ lie in $b$, then all constituents of $R_{\bL_1}^{\bG_1}(\la_1)$
   lie in $c$.
 \end{itemize}
\end{prop}

\begin{proof}
Note that the hypothesis of part~(a) means that for any parabolic subgroup
$\bP$ of $\bG_1$ containing $\bL_1$ as Levi subgroup all constituents of
$R_{\bL_1 \subset P}^{\bG_1}(\la_1)$ lie in $c$. A similar remark applies
to the conclusion, as well as to part~(b).

For (a), note that by the definition of $\Psi_{\bG_1}^\bG$ we have that
all constituents of
$$\eps_\bL\eps_{\bL_1}\RLG(\la)=R_{\bL_1}^\bG(\la_1)
   =R_{\bG_1}^\bG R_{\bL_1}^{\bG_1}(\la_1)$$
are contained in $R_{\bG_1}^\bG(c)$, hence in $b$ by Bonnaf\'e--Rouquier
correspondence. \par
In (b), suppose that $\eta$ is a constituent of $R_{\bL_1}^{\bG_1}(\la_1)$ not
lying in $c$. Then by Bonnaf\'e--Rouquier, $R_{\bG_1}^\bG(\eta)$ does not
belong to $b$, whence $R_{\bL_1}^\bG(\la_1)$ has a constituent not lying in
$b$, contradicting our assumption that all constituents of
$R_{\bL_1}^\bG(\la_1)=\RLG R_{\bL_1}^\bL(\la_1)=\eps_\bL\eps_{\bL_1}\RLG(\la)$
are in~$b$.
\end{proof}

\subsection{$e$-quasi-centrality}
For a prime $\ell$ not dividing $q$, we denote by $\cE(\bG^F,\ell')$ the set
of irreducible characters of $\bG^F$ lying in a Lusztig series $\cE(\bG^F,s)$,
where $s\in\bG^{*F}$ is a semisimple $\ell'$-element. Recall from
\cite[Def.~2.4]{KM} that a character $\chi\in\cE(\bG^F,\ell')$ is said to be of
\emph{central $\ell$-defect} if the $\ell$-block of $\bG^F$ containing $\chi$
has a central defect group and $\chi$ is said to be of
\emph{quasi-central $\ell$-defect} if some (and hence any)
character of $[\bG, \bG]^F$ covered by $\chi$ is of central $\ell$-defect.

\begin{lem}   \label{lem:ext-qcentral}
 Let $\bL$ be an $F$-stable Levi subgroup of $\bG$, and set
 $\bL_0 =\bL\cap[\bG,\bG]$. Let $\ell\ne p$ be a prime.
 \begin{enumerate}
 \item[\rm(a)] If $\bL_0 = C_{[\bG,\bG]}(Z(\bL_0)^F_\ell)$, then
   $\bL=C_\bG(Z(\bL)^F_\ell)$.
 \item[\rm(b)] Let $\la\in\cE(\bL^F,\ell')$ and let $\la_0$ be an irreducible
  constituent of $\Res^{\bL^F}_{\bL_0^F}(\la)$. Then $\la_0$ is of
  quasi-central $\ell$-defect if and only if $\la$ is of quasi-central
  $\ell$-defect.
 \end{enumerate}
\end{lem}

\begin{proof}
Since $\bG = \Zc(\bG)[\bG,\bG]$ and $\Zc(\bG) \leq \bL$, we have that
$\bL = \Zc(\bG)\bL_0$. Hence if $\bL_0=C_{[\bG,\bG]}(Z(\bL_0)^F_\ell)$ then
$\bL=C_\bG(Z(\bL_0)^F_\ell)\supseteq C_\bG(Z(\bL)^F_\ell)\supseteq\bL$.
This proves~(a). In~(b), since $\la$ is in an $\ell'$-Lusztig series, the
index in $\bL^F$ of the stabiliser in $\bL^F$ of $\la_0$ is prime to $\ell$
and on the other hand, $\la_0$ extends to a character of the stabiliser in
$\bL^F$ of $\la_0$. Thus, $\la(1)_\ell = \la_0(1)_\ell$. Since
$[\bL_0,\bL_0]=[\bL,\bL]$, the assertion follows by \cite[Prop.~2.5(a)]{KM}.
\end{proof}

\begin{rem}
The converse of assertion~(a) of Lemma~\ref{lem:ext-qcentral} fails in general,
even when we restrict to $e_\ell(q)$-split Levi subgroups: Let $\ell$ be odd
and $\bG=\GL_\ell$ with $F$ such that $\bG^F=\GL_\ell(q)$ with $\ell |(q-1)$.
Let $\bL$ a $1$-split Levi subgroup of type $\GL_{\ell-1} \times \GL_1$. Then
$Z(\bL)_\ell^F\cong C_\ell\times C_\ell$ and $\bL = C_\bG(Z(\bL)^F_\ell)$.
But $Z(\bL_0)^F_\ell \cong C_\ell \cong Z([\bG, \bG])^F_\ell$, hence
$C_{[\bG,\bG]}(Z(\bL_0)^F_\ell)= [\bG, \bG]$.
\end{rem}

One might hope for further good properties of the bijection of
Proposition~\ref{prop:biject-Jcusp2} with respect to (quasi)-centrality. In
this direction, we observe the following:

\begin{lem}   \label{lem:biject-central}
 In the situation of Proposition~\ref{prop:biject-Jcusp}, if $(\bL,\la)$ is of
 central $\ell$-defect for a prime $\ell$ with $e_\ell(q)=e$, then so is
 $(\bL_1,\la_1)={}^*\Psi_{\bG_1}^\bG(\bL,\la)$, and we have
 $Z(\bL)_\ell^F=Z(\bL_1)_\ell^F$.
\end{lem}

\begin{proof}
By assumption, we have that $\la(1)_\ell=|\bL^F:Z(\bL)^F|_\ell$. Now
$Z(\bL)$ lies in every maximal torus of $\bL$, hence in $\bL_1$, so we have
that $Z(\bL)_\ell^F\le Z(\bL_1)_\ell^F$. As
$\la=\eps_{\bL_1}\eps_\bL R_{\bL_1}^\bL(\la_1)$ we
obtain $\la(1)_\ell=\la_1(1)_\ell|\bL^F:\bL_1^F|_\ell$, whence
$$\la_1(1)_\ell=\la(1)_\ell|\bL^F:\bL_1^F|_\ell^{-1}
  =|\bL_1^F|_\ell|Z(\bL)^F|_\ell^{-1}
  \ge|\bL_1^F:Z(\bL_1)^F|_\ell.$$
But clearly $\la_1(1)_\ell\le |\bL_1^F:Z(\bL_1)^F|_\ell$, so we have equality
throughout, as claimed.
\end{proof}

\begin{exmp}
The converse of Lemma~\ref{lem:biject-central} does not hold in general. To see
this, let $\bG=\PGL_\ell$ with $\bG^F=\PGL_\ell(q)$, $\bL=\bG$, and
$\bG_1\le\bG$ an $F$-stable maximal torus such that $\bG_1^F$ is a Coxeter
torus of $\bG^F$, of order $\Phi_\ell$. Assume that $\ell|(q-1)$ (so $e=1$).
Then $\bL_1=\bG_1$. Here, any $\la_1\in\Irr(\bL_1^F)$ is
$e$-(Jordan-)cuspidal, and certainly of central $\ell$-defect, and
$|Z(\bL_1)_\ell^F|=(\Phi_\ell)_\ell=\ell$ for $\ell\ge3$, while clearly
$Z(\bL)_\ell^F=Z(\bG)_\ell^F=1$. Furthermore
$$\la(1)_\ell= \la_1(1)_\ell [\bL^F: \bL_1^F]_\ell =[\bL^F: \bL_1^F]_\ell$$
since $\la_1$ is linear. Since $|Z(\bL^F)|_\ell =1$ and $|\bL_1^F|_\ell>1$,
it follows that
$$\la(1)_\ell|Z(\bL^F)|_\ell<|\bL^F|_\ell $$
hence $\la$ is not of central $\ell$-defect (and even not of quasi-central
$\ell$-defect).
\end{exmp}

\begin{exmp}
We also recall that $e$-(Jordan-)cuspidal characters are not always of central
$\ell$-defect, even when $\ell$ is a good prime: Let $\bG^F=\SL_{\ell^2}(q)$
with $\ell|(q-1)$, so $e=1$. Then for $\bT$ a Coxeter torus and
$\theta\in\Irr(\bT^F)$ in general position, $\RTG(\theta)$ is
$e$-(Jordan-)\ cuspidal but not of quasi-central $\ell$-defect.
\end{exmp}

For the next definition note that the property of being of (quasi)-central
$\ell$-defect is invariant under automorphisms of $\bG^F$.

\begin{defn} \label{defn:ejqcc}
 Let $\ell\ne p$ be a prime and $e=e_\ell(q)$.
 A character $\chi\in\cE(\bG^F,\ell')$ is called \emph{$e$-Jordan
 quasi-central cuspidal} if $\chi $ is $e$-Jordan cuspidal and the
 $C_{\bG^*}(s)^{F}$-orbit of unipotent characters of $C_{\bG^*}^\circ(s)^F$
 which corresponds to $\chi$ under Jordan decomposition consists of characters
 of quasi-central $\ell$-defect, where $s\in\bG^{*F}$ is a semisimple
 $\ell'$-element such that $\chi\in\cE(\bG^F,s)$.
 An \emph{$e$-Jordan quasi-central cuspidal pair of $\bG$ } is a pair
 $(\bL, \la)$ such that $\bL$ is an $e$-split Levi subgroup of $\bG$ and
 $\la\in\cE(\bL^F,\ell')$ is an $e$-Jordan quasi-central cuspidal character
 of $\bL^F$.
\end{defn}

We note that the set of $e$-Jordan quasi-central cuspidal pairs of $\bG$ is
closed under $\bG^F$-conjugation. Also, note that Lemma~\ref{lem:derived-Jcusp}
remains true upon replacing the $e$-Jordan-cuspidal property by the $e$-Jordan
quasi-central cuspidal property. This is because, with the notation of
Lemma~\ref{lem:derived-Jcusp}, the orbit of unipotent characters corresponding
to $\la$ under Jordan decomposition is a subset of the orbit of unipotent
characters corresponding to $\la_0$ under Jordan decomposition.
Finally we note that the bijection $\Psi_{\bG_1}^\bG$ of
Proposition~\ref{prop:biject-Jcusp2} preserves $e$-quasi-centrality since
with the notation of the proposition $\la_1$ and $\la$ correspond to the
same orbit of unipotent characters under Jordan decomposition.

\section{Lusztig induction and $\ell$-blocks}   \label{sec:main}

Here we prove our main results on the parametrisation of $\ell$-blocks in
terms of $e$-Harish-Chandra series, in arbitrary Levi subgroups of simple
groups of simply connected type. As in Section~\ref{sec:cusp}, $\ell\ne p$
will be prime numbers, $q$ a power of $p$ and and $e=e_\ell(q)$.

\subsection {Preservation of $\ell$-blocks by Lusztig induction}   \label{subsec:all e-splits}

We first extend \cite[Thm.~2.5]{CE99}. The proof will require
three auxiliary results:

\begin{lem}   \label{lem:cabanes-new}
 Let $\bG$ be connected reductive with a Frobenius endomorphism $F$ endowing
 $\bG$ with an $\FF_q$-rational structure. Let $\bM$ be an $e$-split Levi of
 $\bG^F$ and $c$ an $\ell$-block of $\bM^F$. Suppose that
 \begin{enumerate}
  \item[\rm(1)] the set $\{d^{1,\bM^F}(\mu)\mid
   \mu\in\Irr(c)\cap\cE(\bM^F,\ell')\}$ is linearly independent; and
  \item[\rm(2)] there exists a subgroup $Z\le Z(\bM)^F_\ell$ and a block
   $d$ of $C_\bG^\circ(Z)^F$ such that all irreducible constituents of
   $R_\bM^{C_\bG^\circ(Z)}(\mu)$, where
   $\mu\in\Irr(c)\cap\cE(\bM^F,\ell')$, lie in the block~$d$.
 \end{enumerate}
 Then there exists a block $b$ of $\bG^F$ such that all irreducible
 constituents of $\RMG(\mu)$, where $\mu\in\Irr(c)\cap\cE(\bM^F,\ell')$, lie
 in the block $b$.
\end{lem}

\begin{proof}
We adapt the argument of \cite[Prop.~2.16]{KM}. Let $\chi\in\Irr(\bG^F,\ell')$
be such that $\langle\RMG(\mu),\chi\rangle \ne 0$ for some
$\mu\in\Irr(c)\cap\cE(\bM^F,\ell')$. Then $\langle\mu,\sRMG(\chi)\rangle\ne0$.
In particular, $c.\sRMG (\chi) \ne 0$. All constituents of $\sRMG(\chi)$ lie
in $\cE(\bM^F,\ell')$, so by assumption~(1) it follows that
$d^{1,\bM^F}(c.\sRMG (\chi))\ne0$. Since $d^{1,\bM^F}(c.\sRMG(\chi))$ vanishes
on $\ell$-singular elements of $\bM^F$, we have that
$$ \langle d^{1,\bM^F}(c.\sRMG (\chi)), c.\sRMG (\chi) \rangle
  = \langle d^{1,\bM^F}(c.\sRMG (\chi)),d^{1,\bM^F}(c.\sRMG(\chi))\rangle\ne0.$$
If $\varphi$ and $\varphi'$ are irreducible $\ell$-Brauer characters of
$\bM^F$ lying in different $\ell$-blocks of $\bM^F$, then
$\langle \varphi, \varphi' \rangle =0 $ (see for instance
\cite[Ch.~3, Ex.~6.20(ii)]{NT}). Thus,
$$\langle d^{1,\bM^F}(c.\sRMG(\chi)),c'.\sRMG(\chi)\rangle
  = \langle d^{1,\bM^F} (c.\sRMG (\chi)),d^{1,\bM^F}(c'.\sRMG(\chi)\rangle=0$$
for all blocks $c'$ of $\bM^F$ different from $c$.
So, $\langle d^{1,\bM^F}(c.\sRMG(\chi)),\sRMG(\chi)\rangle\ne 0$ from which
it follows that $\langle d^{1,\bM^F}(\mu'),\sRMG(\chi)\rangle\ne 0$
for some $\mu'\in\Irr(c)\cap\cE(\bM^F,\ell')$.

Continuing as in the proof of \cite[Prop.~2.12]{KM} gives the required result.
Note that  Condition~(1) of \cite[Prop.~2.12]{KM} is not necessarily met as
stated, since $\mu'$ may be different from $\mu$. However, $\mu$ and $\mu'$
are in the same block of $\bM^F$ which is sufficient to obtain the conclusion
of the lemma.
\end{proof}

\begin{lem}   \label{lem:product}
 Let $\bG$ be connected reductive with a Frobenius endomorphism $F$.
 Suppose that $\bG$ has connected centre and $[\bG,\bG]$ is simply connected.
 Let $\bG=\bX \bY$ such that either $\bX$ is an $F$-stable product of
 components of $[\bG,\bG]$ and $\bY$ is the product of the remaining components
 with $Z(\bG)$, or vice versa.
 Suppose further that $\bG^F/\bX^F \bY^F$ is an $\ell$-group. Let $\bN$ be an
 $F$-stable Levi subgroup of $\bY$ and set $\bM =\bX \bN$. Let $c$ be
 an $\ell$-block of $\bM^F$ and let $c'$ be an $\ell$-block of $\bN^F$ covered
 by $c$. Suppose that there exists a block $b'$ of $\bY^F$ such that every
 irreducible constituent of $R_\bN^\bY(\tau)$ where
 $\tau\in\Irr(c')\cap\cE(\bN^F,\ell')$ lies in $b'$. Then there exists a
 block $b$ of $\bG^F$ such that every irreducible constituent of $\RMG(\mu)$
 where $\mu\in\Irr(c)\cap\cE(\bM^F,\ell')$ lies in $b$.
\end{lem}

\begin{proof}
We will use the extension of Lusztig induction to certain disconnected groups
as in \cite[Sec.~1.1]{CE99}. Let $\bG_0=[\bG,\bG] =[\bX,\bX] \times [\bY,\bY]$,
$\bM_0= \bG_0 \cap \bM =[\bX,\bX] \times ([\bY, \bY] \cap \bN)$. Then,
$\bG_0^F\subseteq \bX^F \bY^F$ and $\bM_0^F\subseteq\bX^F\bN^F$. Let $\bT$
be an $F$-stable maximal torus of $\bM$. Since $\bG$ and hence also $\bM$
has connected centre, $\bM = \bM_0^F \bT^F$ and $ \bG^F = \bG_0^F\bT^F$.
Further, $A:=\bX^F \bY^F \cap \bT^F = \bX^F \bN^F \cap \bT ^F$ and
$\bX^F \bY^F = \bG_0^F A = (\bG_0A)^F$, $\bX^F\bN^F =\bM_0^F A = (\bM_0 A)^F$.
As in \cite[Sec.~1.1]{CE99}, we denote by $\cE(\bX^F\bY^F,\ell') $ the set of
irreducible characters of $\bX^F\bY^F $ which appear in the restriction of
elements of $\cE(\bG^F, \ell') $ to $ \bX^F\bY^F $.

Let $\chi\in\cE(\bG^F,\ell')$. Since $\bG^F/\bX^F\bY^F$ is an $\ell$-group,
by \cite[Prop.~1.3(i)]{CE99}, $\Res^{\bG^F}_{\bX^F\bY^F}(\chi)$ is irreducible.
Now if $\chi'\in\Irr(\bG^F)$ has the same restriction to $\bX^F\bY^F$ as
$\chi$, then again since $\bG^F/\bX^F\bY^F$ is an $\ell$-group, either
$\chi' =\chi$ or $\chi'\notin \cE(\bG^F,\ell')$. In other words, restriction
from $\ZZ\cE(\bG^F, \ell')$ to $\ZZ\cE(\bX^F\bY^F,\ell')$ is a bijection.
Similarly, restriction from $\ZZ\cE(\bM^F,\ell')$ to $\ZZ\cE(\bX^F\bN^F,\ell')$
is a bijection.

In particular every block of $\bG^F$ covers a unique block of $\bX^F \bY^F$.
Since $ \bG^F/\bX^F \bY^F$ is an $\ell$-group, there is a bijection (through
covering) between the set of blocks of $\bG^F$ and the set of blocks of
$\bX^F\bY^F$. Hence, by the injectivity of restriction from
$\ZZ\cE(\bG^F,\ell')$ to $\ZZ\cE(\bX^F\bY^F,\ell')$, it suffices to prove that
there is a block $b_0$ of $\bX^F \bY^F$ such that every irreducible constituent
of $\Res^{\bG^F}_{\bX^F\bY^F}\RMG(\mu)$ as $\mu$ ranges over
$\Irr(c)\cap\cE(\bM^F,\ell')$ lies in $b_0$.

Following \cite[Sec.~1.1]{CE99}, we have that
$\Res^{\bG^F}_{\bX^F\bY^F}\RMG= R_{\bM_0 A}^{\bG_0A}\Res^{\bM^F}_{\bX^F\bN^F}$
on $\Irr(\bM^F)$ (where here $R_{\bM_0A}^{\bG_0A}$ is Lusztig induction
in the disconnected setting). Thus, it suffices to prove that there is a block
$b_0$ of $\bX^F \bY^F $ such that every irreducible constituent of
$R_{\bM_0A}^{\bG_0A}\Res^{\bM^F}_{\bX^F\bN^F}(\mu)$ as $\mu$ ranges over
$\Irr(c)\cap\cE(\bM^F,\ell')$ is contained in $b_0$.

By the above arguments applied to $\bM^F$ and $\bX^F\bN^F$, there is a unique
block $c_0$ of $\bX^F \bN^F$ covered by $c$. The surjectivity of restriction
from $\ZZ\cE(\bM^F,\ell')$ to $\ZZ\cE(\bX^F\bN^F,\ell')$ implies that
it suffices to prove that there is a block $b_0$ of $\bX^F\bY^F$ such that
every irreducible constituent of $R_{\bM_0A}^{\bG_0A}(\mu)$ for
$\mu\in\Irr(c_0)\cap\cE(\bX^F\bN^F,\ell')$ is contained in $b_0$.

The group $I:=\{(x,x^{-1})\mid x\in \bX^F \cap \bY^F \} \leq \bX \times \bY$
is the kernel of the multiplication map $\bX^F\times\bY^F\to\bX^F\bY^F$.
Identifying $\bX^F\bY^F$ with $\bX^F\times\bY^F/I$ through multiplication,
$\Irr(\bX^F\bY^F)$ is the subset of $\Irr(\bX^F\times\bY^F)$ consisting of
characters whose kernel contains $I$. Since
$\bX^F \cap \bY^ F \leq \bX \cap \bY \leq Z (\bG) \leq \bM $, $I$ is also
the kernel of the multiplication map $\bX^F\times\bN^F\to\bX^F\bN^F$
and we may identify $\Irr(\bX^F\bY^F)$ with the subset of
$\Irr(\bX^F\times\bN^F)$ consisting of characters whose kernel contains $I$.

Any parabolic subgroup of $\bG_0$ containing $\bM_0$ as Levi subgroup is of
the form $[\bX,\bX]\bP$, where $\bP$ is a parabolic subgroup of $[\bY,\bY]$
containing $\bN\cap[\bY,\bY]$ as Levi subgroup. Let
$\bU:= R_{u}(\bX\bP)=R_u(\bP)\leq[\bY,\bY]$ and denote by $\cL^{-1}(\bU)$
the inverse image of $\bU$ under the Lang map $\bG\to\bG$ given by
$g\mapsto g^{-1}F(g)$.

The Deligne--Lusztig variety associated to $R_{\bM_0 A}^{\bG_0A}$ (with respect
to $\bX\bP$) is $\cL^{-1}(\bU)\cap\bG_0 A$. Since $\bT=(\bT\cap\bM_0)Z(\bG)$,
$\bU$ is normalised by $\bT$ and in particular by $A$. Hence,
$$\begin{aligned}
  \cL^{-1}(\bU)\cap\bG_0A = (\cL^{-1}(\bU)\cap\bG_0)A
    &= [\bX, \bX]^F(\cL^{-1}(\bU)\cap[\bY,\bY]) A\\
  &= [\bX, \bX]^F (A \cap  \bX^F)  (\cL^{-1}(\bU)\cap[\bY,\bY])(A\cap\bY^F).
\end{aligned}$$
For the last equality, note that $A=\bX^F \bY^F\cap\bT
  = (\bX^F \cap \bT)(\bY^F \cap \bT) = (\bX^F \cap A) (\bY^F \cap A)$.
Now, $\cL^{-1}(\bU)\cap\bY = (\cL^{-1}(\bU)\cap[\bY,\bY])\bS^F$ for any
$F$-stable maximal torus $\bS$ of $\bY$. Applying this with $\bS=\bT\cap\bY$,
we have that $(\cL^{-1}(\bU)\cap[\bY,\bY])(A\cap\bY^F) = \cL^{-1}(\bU)\cap\bY$.
Also, $[\bX,\bX]^F(A \cap\bX^F)=\bX^F$. Altogether this gives
$\cL^{-1}(\bU)\cap\bG_0A = \bX^F(\cL^{-1}(\bU)\cap\bY)$.
Further, $\cL^{-1}(\bU)\cap\bY $ is the variety underlying $R_\bN^\bY$
(with respect to the parabolic subgroup $\bP Z(\bG)$). Hence, for any
$\tau_1\in\Irr(\bX^F),\tau_2\in\Irr(\bY^F)$ such that $I$ is in the kernel
of $\tau_1\tau_2$, we have
$$R_{\bM_0A}^{\bG_0A}(\tau_1\tau_2)=\tau_1 R_{\bN}^{\bY}(\tau_2).$$
Further, $\tau_1\tau_2\in\cE(\bX^F\bN^F,\ell')$ if and only if
$\tau_1\in\cE(\bX^F,\ell')$ and $\tau_2\in\cE(\bN^F,\ell')$.

To conclude note that $c'$ is the unique block of $\bN^F$ covered by $c_0$ and
$c_0 = d c'$, where $d$ is a block $\bX^F$. Let $b' $ be the block of
$\bY^F$ in the hypothesis. Then, setting $b_0= d b'$
gives the desired result.
\end{proof}

We will also make use of the following well-known extension of
\cite[Prop.~1.5]{En08}.

\begin{lem}   \label{lem:classical2dis}
 Suppose that $q$ is odd. Let $\bG$ be connected reductive with a Frobenius
 endomorphism $F$.
 Suppose that all components of $\bG $ are of classical type $A$, $B$, $C$ or
 $D$ and that $Z(\bG)/\Zc(\bG) $ is a $2$-group. Let $s\in\bG^{*F}$ be
 semisimple of odd order. Then all elements of $\cE(\bG^F, s)$ lie in the
 same $2$-block of $\bG^F$.
\end{lem}

\begin{proof}
Since $s $ has odd order and $Z(\bG)/\Zc(\bG)$ is a $2$-group,
$C_{\bG^*}(s)$ is connected. On the other hand, since all components of
$\bG^*$ are of classical type and $s $ has odd order, $C_{\bG^*}^\circ(s)$
is a Levi subgroup of $\bG$. Thus, $C_{\bG^*}(s)$ is a Levi subgroup of
$\bG^*$ and by Bonnaf\'e--Rouquier the set of $2$-blocks of $\bG^F$ which
contain a character of $\cE(\bG^F,s)$ is in bijection with the set of
unipotent $2$-blocks of $\bC^F$, where $\bC$ is a Levi subgroup of $\bG$
in duality with $C_{\bG^*}(s)$. Since all components of $\bC$ are also
of classical type, the claim follows by \cite[Prop.~1.5(a)]{En08}.
\end{proof}

We now have the following extension of \cite[Thm.~2.5]{CE99} to all primes.

\begin{thm}   \label{thm:all e-splits}
 Let $\bH$ be a simple algebraic group of simply connected type with a
 Frobenius endomorphism $F:\bH\rightarrow\bH$ endowing $\bH$ with an
 $\FF_q$-rational structure. Let $\bG$ be an $F$-stable Levi subgroup of $\bH$.
 Let $\ell$ be a prime not dividing $q$ and set $e=e_\ell(q)$. Let $\bM$ be
 an $e$-split Levi subgroup of $\bG$ and let $c$ be a block of $\bM^F$. Then
 there exists a block $b$ of $\bG^F$ such that every irreducible constituent
 of $\RMG(\mu)$ where $\mu\in\Irr(c)\cap\cE(\bM^F,\ell')$ lies in $b$.
\end{thm}

\begin{proof}
Suppose that $\dim(\bG)$ is minimal such that the claim of the Theorem does
not hold. Let $s\in\bM^{*F}$ be a semisimple $\ell'$-element such that
$\Irr(c)\cap\cE(\bM^F,\ell')\subseteq \cE(\bM^F,s)$. Then all irreducible
constituents of $\RMG(\mu)$ where $\mu\in\Irr(c)\cap\cE(\bM^F,\ell')$ are
in $\cE(\bG^F,s)$.

First suppose that $s$ is not quasi-isolated and let $\bG_1$ be a proper
$F$-stable Levi subgroup of $\bG$ whose dual contains $C_{\bG^*}(s)$. Let
$\bM^*$ be a Levi subgroup of $\bG^*$ in duality with $\bM$ and set
$\bM_1^* = C_{\bG_1^*}(\Zc (\bM^*)_e)$. Then, as in the proof of
Proposition~\ref{prop:biject-Jcusp}, $\bM_1^* $ is an $e$-split Levi subgroup
of $\bG_1^*$ and letting $\bM_1 $ be the dual of $\bM_1^*$ in $\bG$, $\bM_1$
is an $e$-split Levi subgroup of $\bG_1$. Further, $\bM_1^*\geq C_{\bM^*}(s)$.
Hence there exists a unique block say $c_1$ of $\bM_1^F$ such that
$\Irr(c_1)\cap\cE(\bM_1^F,\ell')\subseteq\cE(\bM_1^F,s)$ and such that
$c_1$ and $c$ are Bonnaf\'e--Rouquier correspondents.

By induction our claim holds for $\bG_1$ and the block $c_1$ of $\bM_1$.
Let $b_1$ be the block of $\bG_1^F$ such that every irreducible constituent
of $R_{\bM_1}^{\bG_1}(\mu)$ where $\mu\in\Irr(c_1)\cap\cE(\bM_1^F,\ell')$ lies
in $b_1$ and let $b$ be the Bonnaf\'e--Rouquier correspondent of $b_1$ in
$\bG^F$.

Now let $\mu\in\Irr(c)\cap\cE(\bM^F, s)$ and let $\chi $ be an irreducible
constituent of $\RMG(\mu)$. Let $\mu_1$ be the unique character in
$\Irr(\bM_1^F, s)$ such that $\mu=\pm R_{\bM_1} ^\bM(\mu_1)$. Then,
$\mu_1\in\Irr(c_1)$ and
$$\RMG(\mu)=\RMG(R_{\bM_1}^\bM(\mu_1))
  =R_{\bG_1}^\bG(R_{\bM_1}^{\bG_1}(\mu_1)).$$
All irreducible constituents of $R_{\bM_1}^{\bG_1}(\mu_1)$ lie in $b_1$.
Hence, by the above equation and by the Bonnaf\'e--Rouquier theorem,
$\chi$ lies in $b$, a contradiction.

So, we may assume from now on that $s$ is quasi-isolated in $\bG^*$. By
\cite[Thm.~2.5]{CE99}, we may assume that $\ell$ is bad for $\bG$ and hence
for $\bH$. So $\bH $ is not of type $A$. If $\bH$ is of type $B$, $C$ or $D$,
then $\ell=2$ and we have a contradiction by Lemma~\ref{lem:classical2dis}.

Thus $\bH$ is of exceptional type. Suppose that $s=1$. By \cite[Thm.~3.2]{BMM}
$\bG^F$ satisfies an $e$-Harish-Chandra theory above each unipotent $e$-cuspidal
pair $(\bL,\la)$ and by \cite[Thms.~A,A.bis]{En00}, all irreducible constituents
of $R_{\bL} ^\bG(\la)$ lie in the same $\ell$-block of $\bG^F$.

So we may assume that $s \ne 1$. We consider the case that $\bG=\bH $. Then
by \cite[Thm.~1.4]{KM}, $\bG^F$ satisfies an $e$-Harish-Chandra theory above
each $e$-cuspidal pair $(\bL,\la)$ below $\cE(\bG^F,s)$ and by
\cite[Thm.~1.2]{KM}, all irreducible constituents of $R_{\bL}^\bG(\la)$ lie in
the same $\ell$-block of $\bG^F$.

So, we may assume that $\bG$ is proper in $\bH$. If $\bH $ is of type $G_2$,
$F_4$ or $E_6$, then $\ell=2$, all components of $\bG$ are of classical type.
For $G_2$ and $F_4$ we have that $Z(\bH)$ and therefore $Z(\bG)$ is connected.
If $\bH$ is of type $E_6$, since $2$ is bad for $\bG$, $\bG$ has a component
of type $D_n$, $n\geq 4$. By rank considerations, $[\bG,\bG]$ is of type
$D_4$ or $D_5$. Since $|Z(\bH)/\Zc(\bH)|=3$ it follows again that $Z(\bG)$ is
connected. In either case we get a contradiction by
Lemma~\ref{lem:classical2dis}.

So, $\bH$ is of type $E_7$ or $E_8$. Since $\bG$ is proper in $\bH$,
$5$ is good for $\bG$, hence $\ell=3$ or $2$. Also, we may assume that at
least one of the two assumptions of Lemma~\ref{lem:cabanes-new} fails to hold
for $\bG$, $\bM $ and $c$.

Suppose that $\ell=3$. Since $\bG$ is proper in $\bH$ and $3$ is bad for
$\bG$, either $[\bG,\bG]$ is of type
$E_6$, or $\bH$ is of type $E_8$ and $[\bG,\bG]$ is of type $E_6+A_1$ or of
type $E_7$. In all cases, $Z(\bG)$ is connected (note that if $\bH$ is of
type $E_7$, then $[\bG, \bG]$ is of type $E_6$, whence the order of
$Z(\bG)/\Zc(\bG)$ divides both $2$ and $3$). If $\bG =\bM$, there is
nothing to prove, so we may assume that $\bM$ is proper in $\bG$.
Let $\bC:=C_\bG^\circ(Z(\bM)^F_3)\ge\bM$.

We claim that there is a block, say $d$ of $\bC^F$ such that for all
$\mu\in\Irr(c)\cap\cE(\bM^F,\ell')$, every irreducible constituent of
$R_\bM^\bC(\mu)$ lies in $d$.  Indeed, since $\bM $ is proper in $\bG$ and
since $Z(\bG)$ is connected, by \cite[Prop.~2.1]{CE93} $\bC$ is proper in
$\bG$. Also, by direct calculation either $\bC$ is a Levi subgroup of $\bG$
or $3$ is good for $\bC$. In the first case, the claim follows by the
inductive hypothesis since $\bM$ is also $e$-split in $\bC$.
In the second case, we are done by \cite[Thm.~2.5]{CE99}.

Thus, we may assume that assumption~(1) of Lemma~\ref{lem:cabanes-new} does
not hold. Hence, by \cite[Thm.~1.7]{CE99}, $3$ is bad for $\bM$. Consequently,
$\bM$ has a component of non-classical type. Since $\bM$ is proper in $\bG$,
this means that $[\bG,\bG]$ is of type $E_6+A_1$ or of type $E_7$ and
$[\bM,\bM]$ is of type $E_6$.
Suppose that $[\bG,\bG]$ is of type $E_6+A_1$. Since $[\bM, \bM]$ is of type
$E_6$, and since $3$ is good for groups of type $A$, the result follows from
Lemma~\ref{lem:product}, applied with $\bX$ being the component of $\bG$ of
type $E_6$, and \cite[Thm.~ 2.5]{CE99}.

So we have $[\bG,\bG]$ of type $E_7$ and $[\bM,\bM]$ of type $E_6$. Suppose
that $s$ is not quasi-isolated in $\bM^*$. Then $c$ is in
Bonnaf\'e--Rouquier correspondence with a block, say $c'$ of a proper
$F$-stable Levi subgroup, say $\bM'$ of $\bM$. The prime $3$ is good for any
proper Levi subgroup of $\bM$, hence by \cite [Thm.~1.7]{CE99}
condition~(1) of Lemma~\ref{lem:cabanes-new} holds for the group $\bM'$ and the
block $c'$. By Bonnaf\'e--Rouquier, this condition also holds for $\bM$ and
$c$, a contradiction. So, $s$ is quasi-isolated in $\bM^*$. Since as pointed
out above, $\bG$ has connected center, so does $\bM$ whence $s$ is isolated
in $\bM^*$. Also, note that since $s$ is also quasi-isolated in $\bG^*$, by
the same reasoning $s$ is isolated in $\bG^*$. Inspection shows that the
only possible case for this is when $s$ has order
three with $C_{\bG^*}(s)$ of type $A_5+A_2$, $C_{\bM^*}(s)$ of type $3A_2$.
Since $s$ is supposed to be a $3'$-element, this case does not arise here.

Now suppose that $\ell=2$. Since $Z(\bH)/\Zc(\bH)$ has order dividing~$2$, by
Lemma~\ref{lem:classical2dis} we may assume that $\bG$ has at least one
non-classical component, that is we are in one of the cases $[\bG,\bG] = E_6$,
or $\bH=E_8$ and $[\bG,\bG]=E_6 +A_1$ or $E_7$.  Again, in all cases, $Z(\bG)$
is connected and consequently $C_{\bG^*}(s)$ is connected and $s$ is isolated.

Suppose first that $[\bG,\bG]=E_7$. We claim that all elements of
$\cE(\bG^F,s)$ lie in the same $2$-block. Indeed, let $\bar s$ be the image
of $s$ under the surjective  map $\bG^*\to[\bG,\bG]^*$ induced by the regular
embedding of $[\bG,\bG]$ in $\bG$.
By \cite[Table~4]{KM}, all elements of $\cE([\bG,\bG]^F,\bar s)$ lie in the
same $2$-block, say $d$ of $[\bG,\bG]^F$. So, any block of $\bG^F$ which
contains a character in $\cE(\bG^F, s)$ covers $d$. By general block
theoretical reasons, there are at most $|\bG^F/[\bG,\bG]^F|_{2'}$ $2$-blocks
of $\bG^F$ covering a given $d$. Now since $s$ is a $2'$-element,
$C_{[\bG,\bG]^*}(\bar s)$ is connected.
Thus, if $\mu \in \cE([\bG, \bG]^F,\bar s)$, then there are
$|\bG^F/[\bG,\bG]^F|_{2'}$ different $2'$-Lusztig series of $\bG^F$ containing
an irreducible character covering $\mu$. Since characters in different
$2'$-Lusztig  series lie in different $2$-blocks, the claim follows.

By the claim above, we may assume that either $[\bG,\bG]=E_6$ or
$[\bG,\bG]=E_6+A_1$. Since $s $ is isolated of odd order in
$\bG^*$, by \cite[Table~1]{KM} all components of $C_{\bG^*}(s)$ are of type
$A_2$ or $A_1$. Consequently, all components of $C_{\bM^*}(s)$ are of type $A$.
Suppose first that $\bM$ has a non-classical component. Then $[\bM,\bM]$ is of
type $E_6 $, and $[\bG,\bG]= E_6 +A_1$. This may be ruled out by
Lemma~\ref{lem:product}, applied with $\bX$ equal to the product of the
component of type $E_6$ with $Z(\bG)$ and $\bY$ equal to the component
of type $A_1$.
\par
So finally suppose that all components of $\bM$ are of classical type. Then,
$C_{\bM^*}(s) = C_{\bM^*}^\circ(s)$ is a Levi subgroup of $\bM$ with all
components of type $A$. Hence,  the first hypothesis of
Lemma~\ref{lem:cabanes-new} holds by the Bonnaf\'e--Rouquier theorem and
\cite[Thm.~1.7]{CE99}.
So, we may assume that the second hypothesis of Lemma~\ref{lem:cabanes-new}
does not hold. Let $\bC:=C_\bG^\circ(Z (\bM^F)_2)$. Since $\bM$ is a proper
$e$-split Levi subgroup of $\bG$, and since $Z(\bG)$ is connected, by
\cite[Prop.~2.1]{CE93} $\bC$ is proper in $\bG$. By induction, we may assume
that $\bC$ is not a Levi subgroup of $\bG$. In particular, the intersection
of $\bC$ with the component of type $E_6$ of $\bG$ is proper in that component
and hence all components of $\bC$ are of type $A$
or $D$. If all components of $\bC$ are of type $A$, then $2$ is good for
$\bC$ and the second hypothesis of Lemma~\ref{lem:cabanes-new} holds by
\cite[Thm.~2.5]{CE99}. Thus we may assume that $\bC$ has a component of
type $D$. Since all components of $\bC$ are classical, by
Lemma~\ref{lem:classical2dis}, we may assume that $Z(\bC)/\Zc(\bC) $
is not a $2$-group and consequently $\bC$ has a component of type $A_n$,
with $n\equiv2\pmod3$. But by the Borel--de Siebenthal algorithm, a group
of type $E_6$ has no subsystem subgroup of type $D_m+ A_n$ with $n\geq1$ and
$m\geq 4$.
\end{proof}

\subsection{Characters in $\ell$-blocks } \label{subsec:allcharacters}
Using the results collected so far, it is now easy to characterise all
characters in $\ell'$-series inside a given $\ell$-block in terms of Lusztig
induction.

\begin{defn}
 As in \cite[1.11]{CE99} (see also \cite[Def.~3.1]{BMM}) for $e$-split Levi
 subgroups $\bM_1,\bM_2$ of $\bG$ and $\mu_i\in\Irr(\bM_i^F)$ we write
 $(\bM_1,\mu_1)\le_e(\bM_2,\mu_2)$ if $\bM_1\le\bM_2$ and $\mu_2$ is a
 constituent of $R_{\bM_1}^{\bM_2}(\mu_1)$ (with respect to some parabolic
 subgroup of $\bM_2$ with Levi subgroup $\bM_1$). We let $\ll_e$ denote the
 transitive closure of the relation $\le_e$.
\end{defn}

As pointed out in \cite[1.11]{CE99} it seems reasonable to expect that the
relations $\le_e$ and $\ll_e$ coincide. While this is known to hold for
unipotent characters (see \cite[Thm.~3.11]{BMM}), it is open in general.

We put ourselves in the situation and notation of Theorem~A.

\begin{thm}   \label{thm:all covered}
 Let $b$ be an $\ell$-block of $\bG^F$ and denote by $\cL(b)$ the set of
 $e$-Jordan-cuspidal pairs $(\bL,\la)$ of $\bG$ such that
 $\Irr(b)\cap\RLG(\la)\ne\emptyset$. Then
 $$\Irr(b)\cap\cE(\bG^F,\ell')
   =\{\chi\in\cE(\bG^F,\ell')\mid \exists\,(\bL,\la)\in\cL(b)\text{ with }
    (\bL,\la)\ll_e(\bG,\chi)\}.$$
\end{thm}

\begin{proof}
Let $b$ be as in the statement and first assume that
$\chi\in\Irr(b)\cap\cE(\bG^F,\ell')$. If $\chi$ is not $e$-Jordan-cuspidal,
then it is not $e$-cuspidal, so there exists a proper $e$-split
Levi subgroup $\bM_1$ such that $\chi$ occurs in $R_{\bM_1}^\bG(\mu_1)$ for
some $\mu_1\in\cE(\bM_1^F,\ell')$. Thus inductively we obtain a chain of
$e$-split Levi subgroups $\bM_r\lneq\ldots\lneq\bM_1\lneq\bM_0:=\bG$ and
characters $\mu_i\in\cE(\bM_i^F,\ell')$ (with $\mu_0:=\chi$) such that
$(\bM_r,\mu_r)$ is $e$-Jordan cuspidal and such that
$(\bM_i,\mu_i)\le_e(\bM_{i-1},\mu_{i-1})$ for $i=1,\ldots,r$, whence
$(\bM_r,\mu_r)\ll_e(\bG,\chi)$. Let $b_r$ be the $\ell$-block of $\bM_r^F$
containing $\mu_r$. Now Theorem~\ref{thm:all e-splits} yields that for each
$i$ there exists a block, say $b_i$, of $\bM_i^F$ such that all constituents of
$R_{\bM_i}^{\bM_{i-1}}(\zeta_i)$ lie in $b_{i-1}$ for all
$\zeta_i\in\Irr(b_i)\cap\cE(\bM_i^F,\ell')$. In particular, $\chi$ lies in
$b_0$, so $b_0=b$, and thus $(\bM_r,\mu_r)\in\cL(b)$.
\par
For the reverse inclusion, let $(\bL,\la)\in\cL(b)$ and
$\chi\in\Irr(\bG^F,\ell')$ such that $(\bL,\la)\ll_e(\bG,\chi)$. Thus there
exists a chain of $e$-split Levi subgroups $\bL=\bM_r\lneq\ldots\lneq\bM_0=\bG$
and characters $\mu_i\in\Irr(\bM_i^F)$ with
$(\bM_i,\mu_i)\le_e(\bM_{i-1},\mu_{i-1})$.
Again, application of Theorem~\ref{thm:all e-splits} allows to conclude that
$\chi\in \Irr(b)$.
\end{proof}

\subsection{$\ell$-blocks and derived subgroups}

In the following two results, which will be used in showing that the map
$\Xi$ in Theorem~A is surjective, $\bG$ is connected reductive with Frobenius
endomorphism $F$, and $\bG_0:=[\bG,\bG]$. Here, in the cases that the Mackey
formula is not known to hold we assume that $R_{\bL_0}^{\bG_0}$ and $\RLG$
are with respect to a choice of parabolic subgroups $\bP_0\ge\bL_0$ and
$\bP\ge\bL$ such that $\bP_0 = \bG_0\cap\bP$.

\begin{lem}   \label{lem:commutingR}
 Let $b$ be an $\ell$-block of $\bG^F$ and let $b_0$ be an $\ell$-block of
 $\bG_0^F$ covered by $b$. Let $\bL$ be an $F$-stable Levi subgroup of $\bG$,
 $\bL_0 =\bL\cap \bG_0$
 and let $\la_0\in\Irr(\bL_0^F)$. Suppose that every irreducible constituent
 of $R_{\bL_0}^{\bG_0}(\la_0)$ is contained in $b_0$. Then there exists
 $\la\in\Irr(\bL^F)$ and $\chi\in\Irr(b)$ such that $\la_0$ is an
 irreducible constituent of $\Res^{\bL^F}_{\bL_0^F}(\la)$ and $\chi$ is an
 irreducible constituent of $\RLG(\la)$.
\end{lem}

\begin{proof}
Since $\bG = \Zc(\bG)\bG_0$, by \cite[Prop.~10.10]{B06} we have that
$$\RLG\Ind_{\bL_0^F}^{\bL^F}(\la_0)
  = \Ind_{\bG_0^F}^{\bG^F}R_{\bL_0}^{\bG_0}(\la_0).$$
Note that the result in \cite{B06} is only stated for the case that $\bG$ has
connected centre but the proof does not use this hypothesis.
The right hand side of the above equality evaluated at $1$ is non-zero. Let
$\chi'\in\Irr(\bG^F)$ be a constituent of the left hand side of the equality.
There exists $\la\in\Irr(\bL^F)$ and $\chi_0$ in $\Irr(\bG_0^F)$ such
that $\la$ is an irreducible constituent of $\Ind^{\bL^F}_{\bL_0^F}(\la_0)$,
$\chi'$ is an irreducible constituent of $\RLG(\la)$, $\chi_0$ is an
irreducible constituent of $R_{\bL_0}^{\bG_0}(\la_0)$ and $\chi'$ is
an irreducible constituent of $\Ind_{\bG_0^F}^{\bG^F}(\chi_0)$.
Since $\chi_0\in\Irr(b_0)$, $\chi'$ lies in a block, say $b'$, of $\bG^F$
which covers $b_0$. Since $b$ also covers $b_0$ and since $\bG^F/\bG_0^F$
is abelian, there exists a linear character, say $\theta$ of $\bG^F/\bG_0^F$
such that $b= b'\otimes\theta$ (see \cite[Lemma~2.2]{KM}). Now the result
follows from \cite[Prop.~10.11]{B06} with $\chi =\chi' \otimes \theta$.
\end{proof}

\begin{lem}   \label{lem:commutingRreverse}
 Let $b$ be an $\ell$-block of $\bG^F$ and let $\bL$ be an $F$-stable Levi
 subgroup of $\bG$ and $\la \in \Irr(\bL^F) $ such that every irreducible
 constituent of $R_{\bL}^{\bG}(\la)$ is contained in $b$. Let
 $\bL_0 =\bL\cap \bG_0$ and let $\la_0\in\Irr(\bL_0^F)$ be an irreducible
 constituent of $\Res^{\bL^F}_{\bL_0^F}(\la)$. Then there exists an
 $\ell$-block $b_0$ of $\bG_0^F $ covered by $b$ and an irreducible character
 $\chi_0$ of $\bG_0^F$ in the block $b_0$ such that $\chi_0$ is a constituent
 of $R_{\bL_0}^{\bG_0}(\la_0)$.
\end{lem}

\begin{proof}
Arguing as in the proof of Lemma~\ref{lem:commutingR}, there exists
$\chi\in\Irr(\bG^F)$, $\la'\in\Irr(\bL^F)$ and $\chi_0$ in $\Irr([\bG,\bG]^F)$
such that $\la'$ is an irreducible constituent of
$\Ind^{\bL^F}_{\bL_0^F}(\la_0)$, $\chi$ is an irreducible constituent of
$\RLG(\la')$, $\chi_0$ is an irreducible constituent of
$R_{\bL_0}^{[\bG,\bG]}(\la_0)$ and $\chi$ is an irreducible constituent of
$\Ind_{[\bG,\bG]^F}^{\bG^F}(\chi_0)$. Now, $\la =\theta \otimes  \la' $ for
some linear character $\theta$ of $\bL^F/\bL_0^F$.  By \cite[Prop.~10.11]{B06},
$\theta \otimes \chi$ is an irreducible constituent of $\RLG(\la)$, and
therefore $\theta\otimes\chi\in\Irr(b)$.  Further, $\theta\otimes\chi$ is also
a constituent of $\Ind_{[\bG, \bG]^F}^{\bG^F}(\chi_0)$, hence $b$ covers the
block of $[\bG, \bG]^F$ containing $\chi_0$.
\end{proof}

\subsection {Unique maximal abelian normal subgroups}   \label{subsec:umans}
A crucial ingredient for proving injectivity of the map in parts~(d) and~(e)
of Theorem~A is a property related to the non-failure of factorisation
phenomenon of finite group theory, which holds for the defect groups of many
blocks of finite groups of Lie type and which was highlighted by Cabanes
\cite{C94}: For a prime $\ell$ an $\ell$-group is said to be \emph{Cabanes}
if it has a unique maximal abelian normal subgroup. 

Now first consider the following setting: Let $\bG$ be connected reductive.
For $i=1,2,$ let $\bL_i$ be an $F$-stable Levi subgroup of $\bG$ with
$\la_i\in\cE(\bL_i^F,\ell')$, and let $u_i$ denote the $\ell$-block of $\bL_i^F$
containing $\la_i$. Suppose that $C_\bG(Z(\bL_i^F)_\ell) =\bL_i$ and that
$\la_i$ is of quasi-central $\ell$-defect. Then by
\cite[Props.~2.12, 2.13, 2.16]{KM} there exists a block $b_i$ of $\bG^F$ such
that all irreducible characters of $R_{\bL_i}^\bG (\la_i)$ lie in $b_i$
and $(Z(\bL)_i,u_i) $ is a $b_i$-Brauer pair.

\begin{lem}   \label{lem:cabanesgroup-use}
 In the above situation, assume further that for $i=1,2$ there exists a maximal
 $b_i$-Brauer pair $(P_i,c_i)$ such that $(Z(\bL_i^F)_\ell,u_i)\unlhd(P_i,c_i)$
 and such that $P_i$ is Cabanes. If $b_1=b_2$ then the pairs $(\bL_1,\la_1)$
 and $(\bL_2,\la_2)$ are $\bG^F$-conjugate.
\end{lem}

\begin{proof}
Suppose that $b_1=b_2$. Since maximal $b_1$-Brauer pairs are $\bG^F$-conjugate
it follows that $\,^g(Z(\bL_2^F)_\ell,u_2))\leq\,^g(P_2, c_2) =(P_1,c_1)$ for
some $g\in\bG^F$. By transport of structure, $\,^gZ(\bL_2^F)_\ell$ is a
maximal normal abelian subgroup of $P_1$, hence
$\,^gZ(\bL_2^F)_\ell = Z(\bL_1^F)_\ell$. By the uniqueness of inclusion of
Brauer pairs it follows that
$\,^g(Z(\bL_2^F)_\ell, u_2) =(Z(\bL_1)^F_\ell, u_1)$.
Since $\bL_i =C_\bG(Z(\bL_i^F)_\ell)$ this means that $\,^g\bL_2=\bL_1$.
Further, since $\la_i$ is of quasi central $\ell$-defect, by
\cite[Prop.~2.5(f)]{KM}, $\la_i$ is the unique element of
$\cE(\bL_i^F, \ell')\cap\Irr(u_i)$. Thus $\,^gu_2 =u_1$ implies that
$\,^g\la_2=\la_1$ and $(\bL_1,\la_1)$ and $(\bL_2,\la_2)$ are $\bG^F$-conjugate
as required.
\end{proof}

By the proof of Theorems~4.1 and~4.2 of \cite{CE99} we also have:

\begin{prop}   \label{prop:cabanes-ok}
 Let $\bG$ be connected reductive with simply connected derived subgroup.
 Suppose that $\ell \geq 3$ is good for $\bG$, and $\ell\ne 3$ if $\bG^F$
 has a factor $\tw3D_4(q)$. Let $b$ be an $\ell$-block of $\bG^F$ such that
 the defect groups of $b$ are Cabanes. If $(\bL, \la)$ and $(\bL',\la')$ are
 $e$-Jordan-cuspidal pairs of $\bG$ such that $\la\in\cE(\bL^F,\ell')$,
 $\la'\in\cE(\bL^{'F},\ell')$ with
 $b_{\bG^F}(\bL,\la)=b=b_{\bG^F}(\bL',\la')$, then $(\bL,\la)$ and
 $(\bL',\la')$ are $\bG^F$-conjugate.
\end{prop}

\begin{proof}
This is essentially contained in Section~4 of \cite{CE99}. Indeed, let
$(\bL,\la)$ be an $e$-Jordan-cuspidal pair of $\bG$ such that
$\la\in \cE(\bL^F,\ell')$.
Let $\bT^*$, $\bT$, $\bK=C_\bG^\circ(Z(\bL)_\ell^F)$, $\bK^*$, $\bM$ and
$\bM^*$ be as in the notation before Lemma~4.4 of \cite{CE99}.
Let $Z=Z(\bM)^F_\ell$ and let $\la_\bK$ and $\la_\bM$ be as in
Definition~4.6 of \cite{CE99}, with $\la$ replacing $\zeta$. Then $Z\leq\bT$
and by Lemma~4.8, $\bM = C_\bG^\circ(Z)$. The simply connected hypothesis
and the restrictions on $\ell$ imply that $C_\bG(Z) =C_\bG^\circ (Z) = \bM$.
Let $b_Z=\hat b_Z$ be the $\ell$-block of $\bM^F$ containing $\la_\bM$. Then
by Lemma~4.13, $(Z, b_Z)$ is a self centralising Brauer pair and
$(1,b_{\bG^F}(\bL,\la))\leq(Z,b_Z)$. Further, by Lemma~4.16 of \cite{CE99}
there exists a maximal $b$-Brauer pair $(D,b_D)$ such that
$(Z,b_Z)\leq (D, b_D)$, $Z$ is normal in $D$ and $C_D(Z)=Z$. Note that the
first three conclusions of Lemma~4.16 of \cite{CE99} hold under the conditions
we have on $\ell$ (it is only the fourth conclusion which requires
$\ell\in \Gamma(\bG, F)$). By Lemma~4.10 and its proof, we also have
$$(1,b_{\bG^F}(\bL,\la))\leq (Z(\bL)_\ell^F,b_{\bK^F}(\bL,\la))\leq (Z,b_Z).$$
Suppose that $\bN$ is a proper $e$-split Levi subgroup of $\bG$ containing
$C_\bG^\circ(z) =C_\bG(z)$ for some $1\ne z\in Z(D)\bG_\ba\cap \bG_\bb$.
Then $\bN$ contains $\bL$, $\bM$ and $Z$ by Lemma~4.15(b). Since
$\bL\cap\bG_\bb = \bK\cap\bG_\bb$ by Lemma~4.4(iii), it follows
that $\bN$ also contains $\bK$ and $\bK = C_{\bN}(Z(\bL^F))$. Thus, replacing
$\bG$ with $\bN$ in Lemma~4.13 we get that
$$(1,b_{\bN^F}(\bL,\la))\leq(Z(\bL)^F_\ell,b_{\bK^F}(\bL,\la)) \leq(D,b_D).$$

Let $(\bL',\la')$ be another $e$-Jordan-cuspidal pair of $\bG$ with
$\la'\in\cE(\bL{'^F},\ell')$ such that
$b_{\bG^F}(\bL,\la)=b=b_{\bG^F}(\bL',\la')$. Denote by $\bK',\bM',D'$ etc.\
the corresponding groups and characters for $(\bL',\la')$. Up to replacing
by a $\bG^F$-conjugate, we may assume that $(D',b_{D'}) =(D,b_D)$.

Suppose first that there is a $1\ne z \in Z(D)\bG_\ba \cap \bG_\bb$.
By Lemma~4.15(b), there is a proper $e$-split Levi subgroup $\bN$ containing
$C_\bG(z)$. Moreover, $\bN$ contains $D,\bL',\bM'$, $\bK'$ and $\bG_\ba$
and we also have
$$(1,b_{\bN^F}(\bL',\la'))\leq (Z(\bL')^F_\ell,b_{{\bK'}^F}(\bL',\la'))
  \leq (D,b_D).$$
By the uniqueness of inclusion of Brauer pairs it follows that
$b_{\bN^F}(\bL,\la) = b_{\bN^F}(\bL',\la')$. Also $D$ is a defect group of
$b_{\bN^F}(\bL,\la)$. Thus, in this case we are done by induction.

So, we may assume that $Z(D)\leq \bG_\ba$ hence $D\leq \bG_\ba$.
From here on, the proof of Lemma~4.17 of \cite{CE99} goes through without
change, the only property that is used being that $Z$ is the unique maximal
abelian normal subgroup of $D$.
\end{proof}

We will also need the following observation:

\begin{lem}   \label{lem:cabanes-direct}
 Let $P=P_1 \times P_2$ where $P_1 $ and $P_2 $ are Cabanes. Suppose that
 $P_0$ is a normal subgroup of $P$ such that $\pi_i(P_0)=P_i$, $i=1,2$, where
 $\pi_i:P_1\times P_2\to P_i$ denote the projection maps. Then $P_0$
 is Cabanes with maximal normal abelian subgroup $(A_1\times A_2)\cap P_0$,
 where $A_i$ is the unique maximal normal abelian subgroup of $P_i$, $i=1,2$.
\end{lem}

\begin{proof}
Let $A =A_1 \times A_2$. The group $A \cap P_0$ is abelian and normal in $P_0$.
Let $S$ be a normal abelian subgroup of $P_0$. Since $\pi_i(P_0)=P_i$,
$\pi_i(S)$ is normal in $P_i$ and since $S$ is abelian, so is $\pi_i(S)$.
Thus, $\pi_i(S)$ is a normal abelian subgroup of $P_i$ and is therefore
contained in $A_i$. So,
$S\leq(\pi_1(S)\times\pi_2(S))\cap P_0\leq(A_1\times A_2)\cap P_0=A\cap P_0$
and the result is proved.
\end{proof}

\subsection {Linear and unitary groups at $\ell=3 $}   \label{subsec:An}

The following will be instrumental in the proof of statement~(e) of Theorem~A.

\begin{lem}   \label{lem:SL-non-Cabanes}
 Let $q$ be a prime power such that $3|(q-1)$ (respectively $3|(q+1)$).
 Let $G=\SL_n(q)$ (respectively $\SU_n(q)$) and let $P$ be a Sylow $3$-subgroup
 of $G$. Then $P$ is Cabanes unless $n=3$ and $3||(q-1)$ (respectively
 $3||(q+1)$). In particular, if $P$ is not Cabanes, then $P$ is extra-special
 of order $27$ and exponent $3$. In this case $N_G(P)$ acts transitively on
 the set of subgroups of order $9$ of $P$.
\end{lem}

\begin{proof}
Embed $P\le\SL_n(q)\le\GL_n(q)$. A Sylow $3$-subgroup of $\GL_n(q)$ is
contained in the normaliser $C_{q-1}\wr \fS_n$ of a maximally split torus.
According to \cite[Lemme~4.1]{C94}, the only case in which $\fS_n$ has
a quadratic element on $(C_{q-1}^n)_3\cap\SL_n(q)$ is when $n=3$ and
$3||(q-1)$. If there is no quadratic element in this action, then
$P$ is Cabanes by \cite[Prop.~2.3]{C94}. In the case of $\SU_n(q)$, the same
argument applies with the normaliser $C_{q+1}\wr \fS_n$ of a Sylow 2-torus
inside $\GU_n(q)$. \par
Now assume we are in the exceptional case. Clearly $|P| =27$.
Let $P_1$, $P_2\le P$ be subgroups of order $9$, and let $u_i\in P_i$ be
non-central. Then $u_i$ is $G$-conjugate to $\diag(1,\zeta,\zeta^2)$,
where $\zeta$ is a primitive $3$rd-root of unity in $\FF_q$ (respectively
$\FF_{q^2}$). In particular, there exists $g\in G$ such that $\,^g u_1=u_2$.
Let $\bar{\ }:G\to G/Z(G)$ denote the
canonical map. Then $\,^{\bar g}(\bar u_1) = \bar u_2$.
Since the Sylow 3-subgroup $\bar P$ of $\bar G$ is abelian, there exists
$\bar h\in N_{\bar G} (\bar P)$ with $\,^{\bar h} (\bar u_1)=\bar u_2$.
Then $h\in N_G(P) $ and $\,^h P_1 = P_2$ as $P_i =\langle Z(G), u_i \rangle$.
\end{proof}

\begin{lem}   \label{lem:doublecentraliser}
 Suppose that $3||n$ and $3||(q-1)$ (respectively $3||(q+1)$). Let $\tbG=\GL_n$,
 $\bG=\SL_n$ and suppose that $\tbG^F=\GL_n(q)$ (respectively $\GU_n(q)$).
 Let $s$ be a semisimple $3'$-element of $\tbG^F$ such that a Sylow
 $3$-subgroup $D$ of $C_{\bG^F}(s)$ is extra-special of order $27$ and let
 $P_1, P_2\leq D$ have order~$9$. There exists
 $g\in N_{\bG^F}(D)\cap C_{\bG^F}(C_{\bG^F}(D))$ such that $\,^g P_1=P_2$.
\end{lem}

\begin{proof}
Set $d=\frac{n}{3}$. Identify $\tbG$ with the group of linear transformations
of an $n$-dimensional $\FF_q $-vector space $V$ with chosen basis
$\{e_{i,r}\mid 1 \leq i\leq d,1\leq r\leq 3\}$. For $g\in\tbG$, write
$a(g)_{i,r,j,s}$ for the coefficient of $e_{i,r}$ in $g(e_{j,s})$.
Let $w\in\tbG$ be defined by $w(e_{i,r}) = e_{i+1,r}$,
$1\leq i\leq d, 1\leq r\leq 3$. For $1\leq i\leq d$ let $V_i$ be the span of
$\{e_{i,1},e_{i,2},e_{i,3}\}$ and $\tbG_i = \GL(V_i)$ considered as a subgroup
of $\tbG$ through the direct sum decomposition $V =\oplus_{1\leq i \leq d}V_i$.

Up to conjugation in $\tbG$ we may assume that $F= \ad_w \circ F_0 $, where
$F_0$ is the standard Frobenius morphism which raises every matrix entry to
its $q$-th power in the linear case, respectively the composition of the
latter  by the transpose inverse map in the unitary case. Note that then each
$\tbG_i$ is $F_0$-stable.

Thus, given the hypothesis on the structure of $D$, we may assume the following
up to conjugation: $s$ has $d$ distinct eigenvalues $\delta_1,\ldots,\delta_d$
with $\delta_{i+1} = \delta_i^q$ (respectively $\delta_i^{-q}$);
$V_i$ is the $\delta_i$-eigenspace of $s$, and
$C_{\tbG}(s) =\prod_{i=1}^d \tbG_i$. Further, $F(\tbG_i) = \tbG_{i+1}$ and
denoting by $\Delta:\tbG_1\to\prod_{i=1}^d\tbG_i$,
$x\mapsto xF(x)\cdots F^{d-1} (x)$, the twisted diagonal map we have
$C_{\tbG^F}(s) = \Delta(\tbG_1^{F^d})$.
Here, $\tbG_1^{F^d} = \tbG_1^{F_0^d}$ is isomorphic to either
$\GL_3(q^d)$ or $\GU_3(q^d)$. Note that $\GU_3(q^d)$ occurs only if $d$ is odd.

Consider $\tbG_1^{F_0}\leq\tbG_1^{F_0^d} $. Let $U_1$ be the Sylow
$3$-subgroup of the diagonal matrices in $\tbG_1^{F_0}$ of determinant~$1$ and
let $\sigma_1\in\tbG_1^{F_0}$ be defined by $\sigma_1(e_{1,r}) = e_{1,r+1}$,
$1\leq r \leq 3$. Then $D_1:= \langle U_1,\sigma_1\rangle$ is a Sylow
$3$-subgroup of $\tbG_1^{F_0}$. Since by hypothesis the Sylow $3$-subgroups of
$C_{\bG^F}(s)$ have order $27$, $D:= \Delta(D_1)$ is a Sylow $3$-subgroup of
$C_{\bG^F}(s)$ with $\Delta(U_1)\cong U_1$ elementary abelian of
order $9$. Note that $\Delta(\sigma_1)(e_{i,r}) = e_{i,r+1}$ for
$1\leq i\leq d$ and $ 1\leq r \leq 3$.

Let $\zeta\in \overline{\FF}_q$ be a primitive $3$rd-root of unity.
Let $u_1\in U_1$ be such that $u_1(e_{1,r}) =\zeta^re_{1,r}$, $1\leq r\leq 3$.
For $ 1\leq r \leq 3$, let $W_r$ be the span of $\{e_{1,r},\ldots,e_{d,r}\}$.
Then $W_r$ is the $\zeta^r$-eigenspace of $\Delta(u_1)$, whence
$$C_{\tbG} (D) \leq C_{\tbG} (\Delta (U_1)) =C_ {\tbG }(\Delta(u_1))
   = \prod_{1\leq r\leq 3} \GL(W_r).$$
Since $\Delta(\sigma_1)(W_r) = W_{r+1}$, and $\Delta(\sigma_1)$ acts on
$C_{\tbG}(\Delta (U_1))$, it follows that $C_{\tbG }(D) = \Delta'(\GL(W_1))$,
where $\Delta':\GL(W_1) \to\prod_{1\leq r\leq 3} \GL(W_r)$,
$x\mapsto x \,^\sigma x \,^{\sigma^2} x$, is the twisted diagonal.

We claim that $\Delta(\tbG_1^{F_0})$ centralises $C_{\tbG }(D)$. Indeed, note
that $g\in\Delta(\tbG_1^{F_0})$ if and only if $a(g)_{i, r, j,s } = 0$ if
$i\ne j$ and $a(g)_{i,r,i,s} = a (F_0^{i-1}(g))_{1,r,1,s} =a(g)_{1,r,1,s}$
for all $i$ and all $r,s$. Also, $h\in C_{\tbG}(D)$ if and only if
$a(h)_{i,r,j,s} = 0$ if $r\ne s$ and $a(h)_{i,r,j,r} =a(h)_{i,1,j,1}$ for
all $i,j$ and all $r$. The claim follows from an easy matrix multiplication.

Let $H= [\tbG_1^{F_0},\tbG_1^{F_0}]$ and note that $D_1\leq H$. By
Lemma~\ref{lem:SL-non-Cabanes} applied to $H$ any two subgroups of $D_1$
of order $9$ are conjugate by an element of $N_H(D_1)$.
The lemma follows from the claim above.
\end{proof}

\subsection{Parametrising $\ell$-blocks}

We can now prove our main Theorem~A, which we restate. Recall the  definition
of $e$-Jordan (quasi-central) cuspidal pairs from the previous section.

\begin{thm}   \label{thm:ecusp-blocks}
 Let $\bH$ be a simple algebraic group of simply connected type with a
 Frobenius endomorphism $F:\bH\rightarrow\bH$ endowing $\bH$ with an
 $\FF_q$-rational structure. Let $\bG$ be an $F$-stable Levi subgroup of $\bH$.
 Let $\ell$ be a prime not dividing $q$ and set $e=e_\ell(q)$.
 \begin{enumerate}[\rm(a)]
  \item For any $e$-Jordan-cuspidal pair $(\bL,\la)$ of $\bG$ such that
   $\la\in\cE(\bL^F,\ell')$, there exists a unique $\ell$-block
   $b_{\bG^F}(\bL,\la)$ of $\bG^F$ such that all irreducible constituents
   of $\RLG(\la)$ lie in $b_{\bG^F}(\bL,\la)$.
  \item The map $\Xi:(\bL,\la)\mapsto b_{\bG^F}(\bL,\la)$ is a
   surjection from the set of $\bG^F$-conjugacy classes of $e$-Jordan-cuspidal
   pairs $(\bL,\la)$ of $\bG$ with $\la \in \cE(\bL^F,\ell')$ to the set
   of $\ell$-blocks of~$\bG^F$.
  \item The map $\Xi$ restricts to a surjection from the set of
   $\bG^F$-conjugacy classes of $e$-Jordan quasi-central cuspidal pairs
   $(\bL,\la)$ of $\bG$ with $\la \in \cE(\bL^F,\ell')$ to the set
   of $\ell$-blocks of~$\bG^F$.
  \item For $\ell\ge3$ the map $\Xi$ restricts to a bijection between the set
   of $\bG^F$-conjugacy classes of $e$-Jordan quasi-central cuspidal pairs
   $(\bL,\la)$ of $\bG$ with $\la \in \cE(\bL^F,\ell')$ and the set of
   $\ell$-blocks of $\bG^F$.
  \item The map $\Xi$ itself is bijective if $\ell\geq 3$ is good for $\bG$,
   and $\ell\ne 3$ if $\bG^F$ has a factor $\tw3D_4(q)$.
 \end{enumerate}
\end{thm}

\begin{rem}   \label{rem:part(e)}
Note that (e) is best possible. See \cite{En00}, \cite{KM} for
counter-examples to the conclusion for bad primes, and \cite[p.~348]{En00} for
a counter-example in the case $\ell= 3$ and $\bG^F=\tw3D_4(q)$.
Counter-examples in the case
$\ell=2$ and $\bG$ of type $A_n$ occur in the following situation. Let
$\bG^F =\SL_n(q)$ with $4|(q+1)$. Then $e=2$ and the unipotent
$2$-(Jordan-) cuspidal pairs of $\bG^F$ correspond to $2$-cores of partitions
of $n-1$ (see \cite[\S3A]{BMM}). On the other hand, by \cite[Thm.~13]{CE93},
$\bG^F$ has a unique unipotent $2$-block. \par
Also, part (d) is best possible as the next example shows.
\end{rem}

\begin{exmp}   \label{exmp:part(d)}
Consider $\bG=\SL_n$ with $n>1$ odd, $\tilde \bG = \GL_n$, and let
$\bG^F=\SL_n(q)$ be such that $q\equiv1\pmod n$ and $4|(q+1)$. Then for
$\ell=2$ we have $e=e_2(q)=2$, and $\FF_q$ contains a primitive $n$-th root
of unity, say $\zeta$. Let
$\tilde s=\diag(1,\zeta,\ldots,\zeta^{n-1})\in\tbG^{*F}$ and let $s$ be its
image in $\bG^*=\PGL_n$. Then $C_{\bG^*}^\circ(s)$ is the maximal $1$-torus
consisting of the image of the diagonal torus of $\tilde \bG^*$.
Thus, $(C_{\bG^*}^\circ(s))_2= 1 = \Zc(\bG^*)_2$.

As $|C_{\bG^*}(s)^F:C_{\bG^*}^\circ(s)^F|=n$ we have $|\cE(\bG^F,s)| = n$,
and all of these characters are $2$-Jordan quasi-central cuspidal.
We claim that all elements of $\cE(\bG^F,s)$ lie in the same $2$-block of
$\bG^F$, so do not satisfy the conclusion of Theorem~\ref{thm:ecusp-blocks}(d).

Let $\tilde \bT$ be a maximal torus of $\tbG$ in duality with
$C_{\tbG^*}(s)$ and let $\tilde\theta\in\Irr(\tilde\bT^F)$ in duality
with $\tilde s$. Let $\bT=\tilde\bT\cap\bG$, and let
$\theta=\tilde\theta|_{\bT^F}$. Since $\tilde s$ is regular,
$\tilde\la:= R_{\tilde \bT}^{\tilde \bG}(\theta)\in \Irr(\tbG^F)$, and
$\cE(\tbG^F,\tilde s)=\{\tilde\la\}$. Further, $\tilde\la$ covers every
element of $\cE(\bG^F,s)$. By \cite[ Prop.~10.10(b*)]{B05},
$$\RTG(\theta) =\Res_{\bG^F}^{\tbG^F}R_{\tilde \bT}^{\tbG}(\tilde\theta)
               =\Res_{\bG^F}^{\tbG^F}(\tilde\la).$$
Thus, every element of $\cE(\bG^F,s)$ is a constituent of $\RTG(\theta)$. On
the other hand, since $\tilde \bT$ is the torus of diagonal matrices, we have
$\bT = C_\bG(\bT_2^F)$ by explicit computation. Hence by
\cite[Props.~2.12, 2.13(1), 2.16(1)]{KM}, all constituents of $\RTG(\theta)$
lie in a single $2$-block of~$\bG^F$.
\end{exmp}

\begin{proof}
Parts~(a) and~(b) are immediate from Theorem~\ref{thm:all e-splits} and
the proof of Theorem~\ref{thm:all covered}.
We next consider Part~(e), where it remains to show injectivity under the
given assumptions. By \cite[Thm.~4.1 and Rem.~5.2]{CE99} only $\ell=3$ and
$\bG$ of (possibly twisted) type $A_n$ remains to be considered. Note that the
claim holds if $3\in\Gamma(\bG,F)$ by \cite[Sec.~5.2]{CE99}. Thus we may
assume that the ambient simple algebraic group $\bH$ of simply connected type
is either $\SL_m$ or $E_6$, and $3\not\in\Gamma(\bG,F)$.
By Proposition~\ref{prop:cabanes-ok} the claim holds for all blocks
whose defect groups are Cabanes.

Let first $\bH=\SL_m$ and $\bG\le\bH$ be an $F$-stable Levi subgroup. As
$3\not\in\Gamma(\bG,F)$ we have $3|(q-1)$ when $F$ is untwisted. We postpone
the twisted case for a moment. Embed $\bH\hookrightarrow\tilde\bH=\GL_m$. Then
$\tbG=\bG Z(\bH)$ is an $F$-stable Levi subgroup of $\tilde\bH$, so has
connected center. Moreover, as $\tilde\bH$ is self-dual, so is its Levi subgroup
$\tbG$. In particular, $3\in\Gamma(\tbG,F)$. Now let $b$ be a 3-block of
$\bG^F$ in $\cE_3(\bG^F,s)$, with $s\in{\bG^*}^F$ a semisimple $3'$-element.
Let $\tilde b$ be a block of $\tbG$ covering $b$, contained in
$\cE_3(\tbG^F,\tilde s)$, where $\tilde s$ is a preimage of $s$ under the
induced map $\tbG^*\rightarrow\bG^*$. Since $3|(q-1)$, $C_{\tbG}(\tilde s)^F$
has a single unipotent 3-block, and so by \cite[Prop.~5.1]{CE99} a Sylow
3-subgroup $\tilde D$ of $C_{\tbG}(\tilde s)^F$ is a defect group of
$\tilde b$. Thus, $D:=\tilde D\cap\bG=\tilde D\cap\bH$ is a defect group
of $b$.  \par
Now $C_{\tbG}(\tilde s)$ is an $F$-stable Levi subgroup of $\tbG$, so also
an $F$-stable Levi subgroup of $\tilde\bH=\GL_m$. As such, it is a direct
product of factors $\GL_{m_i}$ with $\sum_i m_i=m$. Assume that there is more
than one $F$-orbit on the set of factors. Then by Lemma~\ref{lem:cabanes-direct}
the Sylow 3-subgroup $\tilde D$ of $C_{\tbG}(\tilde s)^F$ has the property
that $D=\tilde D\cap\bH$ is 'Cabanes' and we are done. Hence, we may assume
that $F$ has just one orbit on the set of factors of $C_{\tbG}(\tilde s)$.
But this is only possible if $F$ has only one orbit on the set of factors
of $\tbG$. This implies that $\tbG^F\cong\GL_n(q^{m/n})$ and
$\bG^F\cong\SL_n(q^{m/n})$ for some $n|m$.
\par
Exactly the same arguments apply when $F$ is twisted, except that now $3|(q+1)$.
So replacing $q$ by $q^{m/n}$ we may now suppose that $\bG=\SL_n$ with
$3\not\in\Gamma(\bG,F)$. Assume that the defect groups of $b$ are not Cabanes.
Let $(\bL,\la)$ be an $e$-Jordan-cuspidal pair for $b$ with $\la\in\cE(\bL^F,s)$
and let $\tbL = \Zc(\tbG)\bL$. There exists an irreducible character
$\tilde\la$ of $\tbL^F$ covering $\la$, an irreducible constituent
$\tilde\chi$ of $R_{\tbL}^{\tbG}(\tilde\la)$ and an irreducible constituent,
say $\chi$ of $\RLG(\la)$ such that $\tilde\chi$ covers $\chi$.
By Lemma~\ref{lem:derived-Jcusp}, $(\tbL, \tilde\la)$ is $e$-Jordan-cuspidal.
Let $\tilde b$ be the block of $\tbG^F$ associated to $(\tbL,\tilde\la)$,
contained in $\cE_3(\tbG^F,\tilde s)$. So, $\tilde b$ covers $b$.

As seen above $C_{\tbG}(\tilde s)^F$ has a single unipotent 3-block and a Sylow
3-subgroup $\tilde D$ of $C_{\tbG}(\tilde s)^F$ is a defect group of
$\tilde b$ and $D:=\tilde D\cap\bG$ is a defect group of~$b$. Moreover $F$ has
a single orbit on the set of factors of $C_{\tbG}(\tilde s)$. By
Lemma~\ref{lem:SL-non-Cabanes},
$C_{\tbG}(\tilde s)^F = \GL_3(q^{\frac{n}{3}})$  or $\GU_3(q^{\frac{n}{3}})$,
$3$ does not divide $\frac{n}{3}$ and $D$ is extra-special of
order~$27$ and exponent~$3$. Also, $\tbL$ is an $e$-split Levi subgroup
isomorphic to a direct product of $3$ copies of $\GL_{\frac{n}{3}}$.

Let $U= Z(\bL)^F_3$ and let $c$ be the $3$-block of $\bL^F$ containing $\la$.
From the structure of $\tbL$ given above, $|U|=9$ and $\bL=C_\bG(U)$.
Thus, by \cite[Thm.~2.5]{CE99} $(U,c)$ is a $b$-Brauer pair. Let $(D,f)$ be
a maximal $b$-Brauer pair such that $(U,c)\leq(D,f)$.

Let $(\bL',\la')$ be another $e$-Jordan-cuspidal pair for $b$ with
$\la'\in\cE({\bL'}^F, s)$. Let $U' = Z (\bL')^F_3$ and let $c'$ be the
$3$-block of ${\bL'}^F$ containing $\la'$, so $|U'|=9$ and $(U',c')$ is also
a $b$-Brauer pair. Since all maximal $b$-Brauer pairs are $\bG^F$-conjugate,
there exists $h\in\bG^F$ such that $\,^h(U',c')\leq(D,f)$. Thus, $U$ and
$\,^hU'$ are subgroups of order~9 of $D$. By Lemma~\ref{lem:doublecentraliser},
there exists $g\in N_{\bG^F}(D)\cap C_{\bG^F}(C_{\bG^F}(D))$ such that
$\,^{gh}U' = U$. Since $g$ centralises $C_{\bG^F}(D)$, $\,^gf = f$ and since
$g$ normalises $D$, $\,^gD=D$. Hence
$$(U,\,^{gh}c')=\,^{gh}(U',c')\leq\,^g(D,f)=(D,f).$$
By the uniqueness of inclusion of Brauer pairs we get that
$ \,^{gh}(U',c') = (U,c)$. Thus $\,^{gh}\bL'=\bL$ and $\,^{gh}c'=c$.
Since $U$ is abelian of maximal order in $D$, $(U,c)$ is a self-centralising
Brauer pair. In particular, there is a unique irreducible character in $c$
with $U$ in its kernel. Since $\la \in \cE(\bL^F, \ell')$, $U$ is contained
in the kernel of $\la$. Hence $\,^{gh}\la'= \la$ and injectivity is proved
for type $A$.

Finally suppose that $\bH$ is of type $E_6$. By our preliminary reductions we
may assume that $\bG$ has only factors of type $A$ and $3\notin\Gamma(\bG,F)$.
Thus $\bG$ must have at least one factor of type $A_2$ or $A_5$. The remaining
possibilities hence are: $\bG$ is of type $A_5$, $2A_2+A_1$, or $2A_2$.
Note that for $\bG$ of type $2A_2+A_1$, the $A_1$-factor of the derived
subgroup $[\bG,\bG]$ splits off, and that $2A_2$ is a Levi subgroup of $A_5$.
So it suffices to show the claim for Levi subgroups of this particular
Levi subgroup $\bG$ of type $A_5$. Since $\bH$ is simply connected,
$[\bG,\bG]\cong\SL_6$ and thus virtually the same arguments as for the case of
$\bG=\SL_n$ apply. This completes the proof of~(e).

Part~(d) follows whenever $\ell\ge3$ is good for $\bG$, and $\ell\ne3$ if
$\bG^F$ has a factor $\tw3D_4(q)$, since then by~(e) there is a unique
$e$-Jordan-cuspidal pair for any $\ell$-block, and its (unipotent) Jordan
correspondent has quasi-central $\ell$-defect by \cite[Prop.~4.3]{CE94} and
Remark~\ref{rem:KM}. So now assume that either $\ell\ge3$ is bad for $\bG$,
or that $\ell=3$ and $\bG^F$ has a factor $\tw3D_4(q)$.

Note that it suffices to prove the statement for quasi-isolated blocks, since
then it follows tautologically for all others using the Bonnaf\'e--Rouquier
Morita equivalences, Proposition~\ref{prop:biject-Jcusp} and the remarks after
Definition~\ref{defn:ejqcc}. Here note that by Lemma~\ref{lem:biject-weyl} the
bijections of Proposition~\ref{prop:biject-Jcusp} extend to conjugacy classes
of pairs. We first prove surjectivity. For this, by Lemma~\ref{lem:commutingR},
Lemma~\ref{lem:ext-qcentral} and by parts~(a) and (b), we may assume that
$\bG =[\bG,\bG]$. Further, since $[\bG,\bG]$ is simply connected, hence a
direct product of its
components, we may assume that $\bG$ is simple. Then surjectivity for unipotent
blocks follows from \cite[Thms.~A, A.bis]{En00}, while for all other
quasi-isolated blocks it is shown in \cite[Thm.~1.2]{KM} (these also include
the case that $\bG^F=\tw3D_4(q)$).

Now we prove injectivity. If $\bG=\bH$, then the claim for unipotent blocks
follows from \cite[Thms.~A, A.bis]{En00}, while for all other quasi-isolated
blocks it is shown in \cite[Thm.~1.2]{KM} (these also include the case that
$\bG^F=\tw3D_4(q)$). Note that in Table~4 of \cite{KM}, each of the
lines 6, 7, 10, 11, 14 and~20 give rise to two $e$-cuspidal pairs and
so to two $e$-Harish-Chandra series, but each $e$-Jordan cuspidal pair
$(\bL,\la)$ which corresponds to these lines has the Cabanes property of
Lemma~\ref{lem:cabanesgroup-use}, so they give rise to different blocks.

So, we may assume that $\bG \ne \bH$, and thus $\ell=3$. Suppose first that
$\bG^F$ has a factor $\tw3D_4(q)$. Then $\bH$ is of type $E_6$,
$E_7$ or $E_8$, there is one component of $[\bG,\bG]$ of type $D_4$ and all
other components are of type $A$. Denote  by $\bG_2$
the component of type $D_4 $, and by $\bG_1$ the product of the remaining
components with $Z^\circ(\bG)$. We note that $Z(\bG_1)/Z^\circ(\bG_1)$ is a
$3'$-group. Indeed, if $\bH $ is of type $E_7$ or $E_8$, then
$Z(\bG)/Z^\circ(\bG)$ is of order prime to $3$, hence the same is true of
$Z(\bG_1)/Z^\circ(\bG_1)$ and if $\bH$ is of type $E_6$, then
$\bG_1=Z^\circ(\bG)$.

Now, $\bG^F = \bG_1^F\times\bG_2^F$. So, the map
$((\bL_1,\la_1),(\bL_2,\la_2)) \to (\bL_1\bL_2, \la_1\la_2) $ is a bijection
between pairs of $e$-Jordan cuspidal pairs for $\bG_1^F $ and
$\bG_2^F$ and $e$-Jordan cuspidal pairs for $ \bG^F$. The bijection preserves
conjugacy and quasi-centrality. All components of $\bG_1 $ are of type $A$ and
as noted above $3$ does not divide the order of $Z(\bG_1)/Z^\circ(\bG_1) $,
hence by \cite[Sec.~5.2]{CE99} we may assume that $\bG =\bG_2 $,
in which case we are done by \cite{En00} and \cite{KM}.

Thus, $\bG^F$ has no factor $\tw3D_4(q)$. Set
$\bG_0:= [\bG, \bG]$. Since $3$ is bad for $\bG$, and $\bG$ is proper in
$\bH$, we are in one of the following cases: $\bH$ is of type $E_7$ and
$\bG_0$ is simple of type $E_6$, or $\bG$ is of type $E_8$ and $\bG_0 $ is of
type $E_6$, $E_6+A_1$ or $E_7$. In all cases note that $Z(\bG)$ is connected,

Let $s\in \bG^{*F}$ be a quasi-isolated semisimple $3'$-element. Let $\bar s$
be the image of $s$ under the surjection $\bG^*\rightarrow\bG_0^*$. Since
$Z(\bG)$ is connected, $s$ is
isolated in $\bG^*$ and consequently $\bar s$ is isolated in $\bG_0^* $.
In particular, if $\bG_0$ has a component of type $A_1 $, then the projection
of $\bar s $ into that factor is the identity. Since $s$ has
order prime to $3$, this means that if $\bG_0$ has a component of type $E_6$,
then $C_{\bG_0^*}(\bar s)$ is connected. We will use this fact later. Also,
we note here that $\bar s\ne 1$ as otherwise the result would follow from
\cite{En00} and the standard correspondence between unipotent
blocks and blocks lying in central Lusztig series. Finally, we note that by
\cite[Thm.~1.2]{KM} the conclusion of Parts~(a) and~(d) of the theorem
holds for $\bG_0^F$ as all components of $\bG_0$ are of different type
(so $e$ is the same for the factors of $\bG_0^F$ as for $\bG^F$).

Let $b$ be a $3$-block of $\bG^F$ in the series $s$ and $(\bL,\la)$ be an
$e$-Jordan quasi-central cuspidal pair for $b$ such that $s\in\bL^{*F}$ and
$\la\in\cE(\bL^F,s)$. Let $\bL_0= \bL \cap\bG_0$ and let $\la_0$ be an
irreducible constituent of the restriction of $\la$ to $\bL_0^F$. By
Lemma~\ref{lem:commutingRreverse} there exists a  block $b_0$ of $\bG_0^F$
covered by $b$, and such that all irreducible constituents of
$R_{\bL_0}^{\bG_0}(\la_0)$ belong to $b$. By Lemma~\ref{lem:derived-Jcusp}
and the remarks following Definition~\ref{defn:ejqcc}, $(\bL_0,\la_0)$ is an
$e$-Jordan quasi-central cuspidal pair of $\bG_0^F$ for $b_0$.

First suppose that $C_{\bG_0}(\bar s)$ is connected. Then all elements of
$\cE(\bG_0^F,\bar s)$ are $\bG^F$-stable and in particular, $b_0$ is
$\bG^F$-stable. Now let $(\bL',\la')$ be another $e$-Jordan quasi-central
cuspidal pair for $b$. Let $\bL_0'=\bL'\cap \bG_0 $ and $\la_0'$ be an
irreducible constituent of the restriction of $\la'$ to $\bL_0^{'F}$. Then,
as above $(\bL_0', \la_0')$ is an $e$-Jordan quasi-central cuspidal pair for
$b_0$. But there is a unique $e$-Jordan quasi-central cuspidal pair for $b_0$
up to $\bG_0^F$-conjugacy. So, up to replacing by a suitable $\bG_0^F$-conjugate
we may assume that $(\bL_0,\la_0) =(\bL_0',\la_0')$, hence $\bL=\bL'$, and
$\la$ and $\la'$ cover the same character $\la_0=\la_0'$ of
$\bL_0^F= \bL_0{'^F}$.

If $\mu\in\cE(\bG_0^F,\bar s)$, then there are $|\bG^F/\bG_0^F|_{3'}$
different $3'$-Lusztig series of $\bG^F$ containing an irreducible character
covering~$\mu$. Since characters in different $3'$-Lusztig series lie in
different $3$-blocks, there are at least $|\bG^F/\bG_0^F|_{3'}$ different
blocks of $\bG^F$ covering $b_0 $. Moreover, if $b'$ is a block of $\bG^F$
covering $b_0$, then there exists a linear character, say $\theta$ of
$\bG^F/\bG_0^F\cong\bL^F/\bL_0^F$ of $3'$-degree such that
$(\bL,\theta\otimes\la)$ is an $e$-Jordan quasi-central cuspidal pair for
$b'$ and
$\la_0$ appears in the restriction of $\theta \otimes \la$ to $\bL_0^F$. Since
there are at most $| \bL^F/ \bL_0^F|_{3'}= |\bG^F/\bG_0^F|_{3'}$ irreducible
characters of $\bL^F$ in $3'$-series covering $\la_0$, it follows that
$\la =\la'$.

Thus, we may assume that $C_{\bG_0}(\bar s)$ is not connected. Hence, by the
remarks above $\bG_0$ is simple of type $E_7$. Further $\bar s$ corresponds to
one of the lines~5, 6, 7, 12, 13, or~14 of Table~4 of \cite{KM} (note that
$\bar s$ is isolated and that $e$-Jordan (quasi)-central cuspidality in this
case is the same as $e$-(quasi)-central cuspidality).

By \cite[Lemma~5.2]{KM}, $\bL_0=C_{\bG_0} (Z(\bL_0^F)_3)$. In other words,
$(\bL_0,\la_0)$ is a good pair for $b_0$ in the sense of \cite[Def.~7.10]{KM}.
In particular, there is a maximal $b_0$-Brauer pair $(P_0,c_0)$ such that
$(Z(\bL_0^F)_3,b_{\bL_0^F}(\la_0))\unlhd(P_0,c_0)$. Here for a finite group
$X$ and an irreducible character $\eta$ of $X$, we denote by $b_X(\eta)$ the
$\ell$-block of $X$ containing $\eta$. By inspection of  the relevant lines
of Table~4 of \cite{KM} (and the proof of \cite[Thm.~1.2]{KM}), one
sees that the maximal Brauer pair $(P_0,c_0)$ can be chosen so that
$Z(\bL_0^F)_3$ is the unique maximal abelian normal subgroup of $P_0$.

By \cite[Thm.~7.11]{KM} there exists a maximal $b$-Brauer pair $(P,c)$ and
$\nu\in\cE(\bL^F,\ell')$ such that $\nu$ covers $\la_0$, $P_0\leq P$
and we have an inclusion of $b$-Brauer pairs
$(Z(\bL^F)_3,b_{\bL^F}(\nu))\unlhd (P,c)$. Since $\la$ also covers
$\la_0$, $\la=\tau \otimes \nu $ for some linear character $\tau$ of
$\bL^F/\bL_0^F \cong \bG^F/\bG_0^F $. Since tensoring with linear characters
preserves block distribution and commutes with Brauer pair inclusion, replacing
$c$ with the block of $C_{\bG^F} (P_0) $ whose irreducible characters are
of the form $\tau \otimes\varphi$, $\varphi\in\Irr(c)$, we get that there
exists a maximal $b$-Brauer pair $(P,c)$ such that
$P_0\leq P$ and $(Z(\bL^F)_3,b_{\bL^F}(\la))\unlhd (P,c)$.

Being normal in $\bG^F$, $Z(\bG^F)_3$ is contained in the defect groups of
every block of $\bG^F$, and in particular $Z(\bG^F)_3\leq P$. On the other
hand, since $\bG_0$ has centre of order $2$, $P_0Z(\bG^F)_3$ is a defect group
of $b$ whence $P$ is a direct product of $P_0$ and $Z(\bG^F)_3$. Now,
$Z(\bL_0^F)_3$ is the unique maximal abelian normal subgroup of $P_0$. Hence,
$Z(\bL^F)_3= Z(\bG^F)_3\times Z(\bL_0^F)_3$ is the unique maximal normal
abelian subgroup of $P$ (see Lemma~\ref{lem:cabanes-direct}). Finally note
that by Lemma~\ref{lem:ext-qcentral}, $\la$ is also of quasi-central
$\ell$-defect. By Lemma~\ref{lem:cabanesgroup-use} it follows that $(\bL,\la)$
is the unique $e$-Jordan quasi-central cuspidal pair of $\bG^F$ for $b$.

Finally, we show~(c). In view of the part~(d) just proved above, it remains
to consider the prime $\ell=2$ only. Suppose first that all components
of $\bG$ are of classical type. Let $s\in\bG^{*F}$ be semisimple of odd order
and let $b$ be a $2$-block of $\bG^F$ in series $s$. By Lemma~\ref{lem:maxtor}
below there is an $e$-torus, say $\bS$ of $C_{\bG^*}^\circ(s)$ such that
$\bT^*:= C_{C_{\bG^*}^\circ(s)}(\bS)$ is a maximal torus of
$C_{\bG^*}^\circ(s)$. Let $\bL^* = C_{\bG^*}(\bS)$ and let $\bL$ be a Levi
subgroup of $\bG$ in duality with $\bL^*$. Then $\bL$ is an $e$-split subgroup
of $\bG$ and $\bT^* = C_{\bL^*} ^\circ(s)$. Let $\lambda \in\Irr(\bL^F,s)$
correspond via Jordan decomposition to the trivial character of $\bT^{*F}$.
Then $(\bL,\la)$ is an $e$-Jordan quasi-central cuspidal pair of $\bG$.

Let $\bG\hookrightarrow\tbG$ be a regular embedding. By part~(a),
Lemmas~\ref{lem:classical2dis} and~\ref{lem:commutingRreverse}, there exists
$g\in\tbG^F$ such that $b= b_{\bG^F}(\,^g\bL,\,^g\la)$.
Now since $(\bL,\la)$ is $e$-Jordan quasi-central cuspidal, so is
$(\,^g\bL,\,^g\la)$. In order to see this, first note that, up to
multiplication by a suitable element of $\bG^F$  and by an application of the
Lang--Steinberg theorem, we may assume that $g$ is in some $F$-stable maximal
torus of $\Zc (\tbG) \bL$. Thus $\,^g \bL=\bL$, and $\la$ and
$\,^g \la$ correspond to the same $C_{\bL^*}(s)^F$ orbit of unipotent
characters of $C_{\bL^*}^\circ(s)^{F}$.

Now suppose that $\bG$ has a component of exceptional type. Then we can argue
just as in the proof of surjectivity for bad $\ell$ in Part~(d).
\end{proof}

\begin{lem}   \label{lem:maxtor}
 Let $\bG$ be connected reductive with a Frobenius morphism
 $F:\bG\rightarrow\bG$. Let $e\in\{1,2\}$ and let $\bS$ be a Sylow $e$-torus
 of $\bG$. Then $C_\bG(\bS)$ is a torus.
\end{lem}

\begin{proof}
Let $\bC:=[C_\bG(\bS),C_\bG(\bS)]$ and assume that $\bC$ has semisimple rank
at least one. Let $\bT$ be a maximally split torus of $\bC$. Then the Sylow
1-torus of $\bT$, hence of $\bC$ is non-trivial. Similarly, the reductive group
$\bC'$ with complete root datum obtained from that of $\bC$ by replacing the
automorphism on the Weyl group by its negative, again has a non-trivial Sylow
$1$-torus. But then $\bC$ also has a non-trivial Sylow $2$-torus. Thus in any
case $\bC$ has a non-central $e$-torus, which is a contradiction to its
definition.
\end{proof}


\section{Jordan decomposition of blocks } \label{sec:Jordan}
Lusztig induction induces Morita equivalences between Bonnaf\'e--Rouquier
corresponding blocks. We show that this also behaves nicely with respect to
$e$-cuspidal pairs and their corresponding $e$-Harish-Chandra series.

\subsection{Jordan decomposition and $e$-cuspidal pairs}
Throughout this subsection, $\bG$ is a connected reductive algebraic group
with a Frobenius endomorphism $F:\bG\rightarrow\bG$ endowing $\bG$ with an
$\FF_q$-structure for some power $q$ of~$p$. Our results here are valid for
all groups $\bG^F$ satisfying the Mackey-formula for Lusztig induction. At
present this is know to hold unless $\bG$ has a component $\bH$ of type $E_6$,
$E_7$ or $E_8$ with $\bH^F\in\{\tw2E_6(2), E_7(2),E_8(2)\}$, see
Bonnaf\'e--Michel \cite{BM11}. The following is in complete analogy with
Proposition~\ref{prop:biject-Jcusp}:

\begin{prop}   \label{prop:biject-ecusp}
 Assume that $\bG^F$ has no factor $\tw2E_6(2)$, $E_7(2)$ or $E_8(2)$.
 Let $s\in {\bG^*}^F$, and $\bG_1\le\bG$ an $F$-stable Levi subgroup with
 $\bG_1^*$ containing
 $C_{\bG^*}(s)$. For $(\bL_1,\la_1)$ an $e$-cuspidal pair of $\bG_1$
 below $\cE(\bG_1^F,s)$ define $\bL:=C_\bG(\Zc(\bL_1)_e)$ and
 $\la:=\eps_\bL\eps_{\bL_1}R_{\bL_1}^\bL(\la_1)$. Then
 $(\bL_1,\la_1)\mapsto(\bL,\la)$ defines a bijection $\Psi_{\bG_1}^\bG$
 between the set of $e$-cuspidal pairs of $\bG_1$ below $\cE(\bG_1^F,s)$ and
 the set of $e$-cuspidal pairs of $\bG$ below $\cE(\bG^F,s)$.
\end{prop}

\begin{proof}
We had already seen in the proof of Proposition~\ref{prop:biject-Jcusp}
that $\bL$ is $e$-split and $\Zc(\bL_1)_e=\Zc(\bL)_e$. For the
well-definedness of $\Psi_{\bG_1}^\bG$ it remains to show that $\la$ is
$e$-cuspidal. For any $e$-split Levi subgroup $\bX\le\bL$ the
Mackey formula \cite[Thm.]{BM11} gives
$$\eps_\bL\eps_{\bL_1}\sRXL(\la)={}\sRXL R_{\bL_1}^\bL(\la_1)
  =\sum_g R_{\bX\cap{}^g\bL_1}^\bX\,
   {}^*\!R_{\bX\cap{}^g\bL_1}^{^g\bL_1}(\la_1^g)$$
where the sum runs over a suitable set of double coset representatives
$g\in\bL^F$. Here, $\bX\cap{}^g\bL_1$ is $e$-split in $\bL_1$ since
$\bL_1\cap\bX^g=\bL_1\cap C_\bL(\Zc(\bX^g)_e)=C_{\bL_1}(\Zc(\bX^g)_e)$. The
$e$-cuspidality of $\la_1$ thus shows that the only non-zero terms in the
above sum are those for which $\bL_1\cap\bX^g=\bL_1$, i.e., those with
$\bL_1\le\bX^g$. But then $\Zc(\bL)_e=\Zc(\bL_1)_e=\Zc(\bX^g)_e$, and as $\bX$
is $e$-split in $\bL$ we deduce that necessarily $\bX=\bL$ if $\sRXL(\la)\ne0$.
So $\la$ is indeed $e$-cuspidal, and $\Psi_{\bG_1}^\bG$ is well-defined.
\par
Injectivity was shown in the proof of Proposition~\ref{prop:biject-Jcusp},
where we had constructed an inverse map with $\bL_1^*:=\bL^*\cap\bG_1^*$
and $\la_1$ the unique constituent of $^*R_{\bL_1}^\bL(\la)$ in
$\cE(\bL_1^F,s)$. We claim that $\la_1$ is $e$-cuspidal. Indeed, for any
$e$-split Levi subgroup $\bX\le\bL_1$ let $\bY:=C_\bL(\Zc(\bX)_e)$, an
$e$-split Levi subgroup of $\bL$. Then $^*R_\bX^{\bL_1}(\la_1)$ is a
constituent of
$$^*\!R_\bX^\bL(\la)={}^*\!R_\bX^\bY {}^*\!R_\bY^\bL(\la)=0$$
by $e$-cuspidality of $\la$, unless $\bY=\bL$, whence
$\bX=\bY\cap\bL_1=\bL\cap\bL_1=\bL_1$. \par
Thus we have obtained a well-defined map $^*\Psi_{\bG_1}^\bG$ from $e$-cuspidal
pairs in $\bG$ to $e$-cuspidal pairs in $\bG_1$, both below the series $s$.
The rest of the proof is again as for Proposition~\ref{prop:biject-Jcusp}.
\end{proof}

\subsection{Jordan decomposition, $e$-cuspidal pairs and $\ell$-blocks}

We next remove two of the three possible exceptions in
Proposition~\ref{prop:biject-ecusp} for characters in $\ell'$-series:

\begin{lem}   \label{lem:2E6,E7}
 The assertions of Proposition~\ref{prop:biject-ecusp} remain true for $\bG^F$
 having no factor $E_8(2)$ whenever $s\in{\bG^*}^F$ is a semisimple
 $\ell'$-element, where $e=e_\ell(q)$. In particular, $\Psi_{\bG_1}^\bG$
 exists.
\end{lem}

\begin{proof}
Let $s$ be a semisimple $\ell'$-element. Then by \cite[Thm.~4.2]{CE99}
we may assume that $\ell\le3$, so in fact $\ell=3$. The character table of
${\bG^*}^F=\tw2E_6(2).3$ is known; there
are 12 classes of non-trivial elements $s\in{\bG^*}^F$ of order prime to~6.
Their centralisers $C_{\bG^*}(s)$ only have factors of type $A$, and are
connected. Thus all characters in those series $\cE(\bG^F,s)$ are uniform, so
the Mackey-formula is known for them with respect to any Levi subgroup.
Thus, the argument in Proposition~\ref{prop:biject-ecusp} is applicable to
those series. For $\bG^F=E_7(2)$, the conjugacy classes of semisimple elements
can be found on the webpage \cite{Lue} of Frank L\"ubeck. From this one
verifies that again all non-trivial semisimple $3'$-elements have centraliser
either of type $A$, or of type $\tw2D_4(q)A_1(q)\Ph4$, or $\tw3D_4(q)\Ph1\Ph3$.
In the latter two cases, proper Levi subgroups are either direct factors, or
again of type $A$, and so once more the Mackey-formula is known to hold with
respect to any Levi subgroup.
\end{proof}

\begin{rem}
The assertion of Lemma~\ref{lem:2E6,E7} can be extended to most $\ell'$-series
of $\bG^F=E_8(2)$. Indeed, again by \cite[Thm.~4.2]{CE99} we only need to
consider $\ell\in\{3,5\}$. For $\ell=3$
there are just two types of Lusztig series for $3'$-elements which can not be
treated by the arguments above, with corresponding centraliser $E_6(2)\Ph3$
respectively $\tw2D_6(2)\Ph4$. For $\ell=5$, there are five types of Lusztig
series, with centraliser $\tw2E_6(2)\tw2A_2(2)$, $E_7(2)\Ph2$,
$\tw2D_7(2)\Ph2$, $E_6(2)\Ph3$ and $\tw2D_5(2)\Ph2\Ph6$ respectively. Note that
the first one is isolated, so the assertion can be checked using \cite{KM}.
\end{rem}

\begin{prop}   \label{prop:biject-ecusp2}
 Assume that $\bG^F$ has no factor $E_8(2)$. Let $s\in {\bG^*}^F$, and
 $\bG_1\le\bG$ an $F$-stable Levi subgroup with $\bG_1^*$ containing
 $C_{\bG^*}(s)$. Assume that $b$ is an $\ell$-block in $\cE_\ell(\bG^F,s)$,
 and $c$ is its Bonnaf\'e--Rouquier correspondent in $\cE_\ell(\bG_1^F,s)$.
 Let $e=e_\ell(q)$.
 \begin{itemize}
  \item[\rm(a)] Let $(\bL_1,\la_1)$ be $e$-cuspidal in $\bG_1$, where
   $(\bL,\la)=\Psi_{\bG_1}^\bG(\bL_1,\la_1)$. If all constituents of
   $R_{\bL_1}^{\bG_1}(\la_1)$ lie in $c$, then all constituents of $\RLG(\la)$
   lie in $b$.
  \item[\rm(b)] Let $(\bL,\la)$ be $e$-cuspidal in $\bG$, where
   $(\bL_1,\la_1)={}^*\Psi_{\bG_1}^\bG(\bL,\la)$. If all constituents of
   $\RLG(\la)$ lie in $b$, then all constituents of $R_{\bL_1}^{\bG_1}(\la_1)$
   lie in $c$.
 \end{itemize}
\end{prop}

The proof is identical to the one of Proposition~\ref{prop:biject-Jcusp2},
using Proposition~\ref{prop:biject-ecusp} and Lemma~\ref{lem:2E6,E7} in
place of Proposition~\ref{prop:biject-Jcusp}.



\begin{thebibliography}{131}

\bibitem{B05}
{\sc C. Bonnaf\'e}, Quasi-isolated elements in reductive groups.
  Comm. Algebra {\bf33} (2005), 2315--2337.

\bibitem{B06}
{\sc C. Bonnaf\'e}, \emph{Sur les caract{\`e}res des groupes r\'eductifs
  finis a centre non connexe : applications aux groupes sp\'eciaux lin\'eaires
  et unitaires}. Ast{\'e}risque {\bf306} (2006).

\bibitem{BM11}
{\sc C. Bonnaf\'e, J. Michel}, Computational proof of the Mackey formula
  for $q>2$. J. Algebra {\bf327} (2011), 506--526.

\bibitem{BR03}
{\sc C. Bonnaf\'e, R. Rouquier}, Cat\'egories d\'eriv\'ees et vari\'et\'es
  de Deligne--Lusztig. Publ. Math. Inst. Hautes \'Etudes Sci. No. {\bf97}
  (2003), 1--59.

\bibitem{BMM}
{\sc M. Brou\'e, G. Malle, J. Michel}, Generic blocks of finite reductive
  groups. Ast\'erisque  No. 212 (1993), 7--92.

\bibitem{C94}
{\sc M. Cabanes},  Unicit\'e du sous-groupe ab\'elian distingu\'e maximal dans
  certains sous-groupes de Sylow. C. R. Acad. Sci. Paris {\bf 318} (1994),
  889--894.

\bibitem{CE93}
{\sc M. Cabanes, M. Enguehard}, Unipotent blocks of finite reductive groups
  of a given type. Math. Z. {\bf 213} (1993), 479--490.

\bibitem{CE94}
{\sc M. Cabanes, M. Enguehard}, On unipotent blocks and their ordinary
  characters. Invent. Math. {\bf 117} (1994), 149--164.

\bibitem{CE99}
{\sc M. Cabanes, M. Enguehard}, On blocks of finite reductive groups and
  twisted induction.  Adv. Math. {\bf145} (1999), 189--229.

\bibitem{DM91}
{\sc F. Digne, J. Michel}, \emph{Representations of Finite Groups of Lie Type}.
 LMS Student Texts, 21. Cambridge University Press, Cambridge, 1991.

\bibitem{En00}
{\sc M. Enguehard}, Sur les $l$-blocs unipotents des groupes r\'eductifs
  finis quand $l$ est mauvais. J. Algebra {\bf230} (2000), 334--377.

\bibitem{En08}
{\sc M. Enguehard}, Vers une d\'ecomposition de Jordan des blocs des groupes
  r\'eductifs finis. J. Algebra {\bf319} (2008), 1035--1115.

\bibitem{KM}
{\sc R. Kessar, G. Malle}, Quasi-isolated blocks and Brauer's height
  zero conjecture. Annals of Math. {\bf178} (2013), 321--384.

\bibitem{Lue}
{\sc F. L\"ubeck}, Table at
  http://www.math.rwth-aachen.de/\~{}Frank.Luebeck/chev/index.html

\bibitem{MT}
{\sc G. Malle, D. Testerman}, \emph{Linear Algebraic Groups and Finite
  Groups of Lie Type}. Cambridge Studies in Advanced Mathematics, 133,
  Cambridge University Press, Cambridge, 2011.

\bibitem{NT}
{\sc H. Nagao, Y. Tsushima}, \emph{Representations of Finite Groups}.
  Academic Press,  Boston, 1989.

\end{thebibliography}
\end{document}